\documentclass[english]{amsart}
\usepackage[T1]{fontenc}
\usepackage[utf8]{inputenc}
\usepackage{geometry}
\usepackage{babel}
\usepackage{refstyle}
\usepackage{amsthm}
\usepackage{amssymb}
\PassOptionsToPackage{normalem}{ulem}
\usepackage{ulem}
\usepackage[unicode=true,pdfusetitle,
 bookmarks=true,bookmarksnumbered=false,bookmarksopen=false,
 breaklinks=false,pdfborder={0 0 1},backref=false,colorlinks=false]
 {hyperref}
\usepackage{breakurl}

\makeatletter

\AtBeginDocument{\providecommand\thmref[1]{\ref{thm:#1}}}
\AtBeginDocument{\providecommand\obsref[1]{\ref{obs:#1}}}
\AtBeginDocument{\providecommand\propref[1]{\ref{prop:#1}}}
\AtBeginDocument{\providecommand\remref[1]{\ref{rem:#1}}}
\AtBeginDocument{\providecommand\factref[1]{\ref{fact:#1}}}
\AtBeginDocument{\providecommand\corref[1]{\ref{cor:#1}}}
\AtBeginDocument{\providecommand\lemref[1]{\ref{lem:#1}}}
\AtBeginDocument{\providecommand\claimref[1]{\ref{claim:#1}}}
\AtBeginDocument{\providecommand\subref[1]{\ref{sub:#1}}}
\AtBeginDocument{}
\RS@ifundefined{subref}
  {\def\RSsubtxt{section~}\newref{sub}{name = \RSsubtxt}}
  {}
\RS@ifundefined{thmref}
  {\def\RSthmtxt{theorem~}\newref{thm}{name = \RSthmtxt}}
  {}
\RS@ifundefined{lemref}
  {\def\RSlemtxt{lemma~}\newref{lem}{name = \RSlemtxt}}
  {}

\numberwithin{equation}{section}
\numberwithin{figure}{section}
\usepackage{enumitem}		
\theoremstyle{plain}
\newtheorem{thm}{\protect\theoremname}[section]
  \theoremstyle{remark}
  \newtheorem{rem}[thm]{\protect\remarkname}
  \theoremstyle{definition}
  \newtheorem{defn}[thm]{\protect\definitionname}
\theoremstyle{plain}
\newtheorem{obs}[thm]{Observation}
  \theoremstyle{plain}
  \newtheorem{fact}[thm]{\protect\factname}
  \theoremstyle{plain}
  \newtheorem{cor}[thm]{\protect\corollaryname}
  \theoremstyle{plain}
  \newtheorem{prop}[thm]{\protect\propositionname}
  \theoremstyle{plain}
  \newtheorem{lem}[thm]{\protect\lemmaname}
  \theoremstyle{remark}
  \newtheorem*{claim*}{\protect\claimname}
  \theoremstyle{remark}
  \newtheorem{claim}[thm]{\protect\claimname}

\usepackage{lmodern}

\usepackage{amsmath}
\usepackage{fullpage}

\newref{claim}{name = Claim~}
\newref{prop}{name = Proposition~}
\newref{cor}{name = Corollary~}
\newref{lem}{name = Lemma~}
\newref{thm}{name = Theorem~}
\newref{rem}{name = Remark~}
\newref{fact}{name = Fact~}
\newref{sub}{name = Section~}
\newref{obs}{name = Observation~}

\newlist{thmlist}{enumerate}{5}
\setlist[thmlist]{label=(\arabic{thmlisti}), ref=\thethm~(\arabic{thmlisti}), noitemsep}

\makeatother

  \providecommand{\claimname}{Claim}
  \providecommand{\corollaryname}{Corollary}
  \providecommand{\definitionname}{Definition}
  \providecommand{\factname}{Fact}
  \providecommand{\lemmaname}{Lemma}
  \providecommand{\propositionname}{Proposition}
  \providecommand{\remarkname}{Remark}
\providecommand{\theoremname}{Theorem}

\begin{document}
\global\long\def\N{\mathbb{N}}
\global\long\def\Z{\mathbb{Z}}
\global\long\def\Q{\mathbb{Q}}
\global\long\def\R{\mathbb{R}}
\global\long\def\C{\mathbb{C}}
 \global\long\def\tg{\prime}
\global\long\def\tgg{\prime\prime}
 \global\long\def\cM{\mathcal{\mathcal{M}}}
\global\long\def\cN{\mathcal{N}}
\global\long\def\cZ{\mathcal{Z}}
\global\long\def\cQ{\mathcal{Q}}

\newcommand{\set}[1]{\left\{ {#1} \right\}}
\newcommand{\abs}[1]{\lvert {#1} \rvert}

\title{A new dp-minimal expansion of the integers}

\author{Eran Alouf and Christian d\textquoteright Elb\'ee}

\date{\today}
\begin{abstract}
We consider the structure $(\mathbb{Z},+,0,|_{p_{1}},\dots,|_{p_{n}})$,
where $x|_{p}y$ means $v_{p}(x)\leq v_{p}(y)$ and $v_p$ is the $p$-adic valuation. We prove that its
theory has quantifier elimination in the language $\{+,-,0,1,(D_{m})_{m\geq1},|_{p_{1}},\dots,|_{p_{n}}\}$ where $D_m(x)\leftrightarrow \exists y ~ my = x$,
and that it has dp-rank $n$. In addition, we prove that a first order
structure with universe $\mathbb{Z}$ which is an expansion of $(\mathbb{Z},+,0)$
and a reduct of $(\mathbb{Z},+,0,|_{p})$ must be interdefinable with
one of them. We also give an alternative proof for Conant's analogous
result about $(\mathbb{Z},+,0,<)$.
\end{abstract}

\maketitle


\section{Introduction}

The study of ``well-behaved'' expansions of $(\Z,+,0)$ is a recent
subject. Until not long ago, no examples of such structures were
studied, other than $(\Z,+,0,<)$. The first examples were given independently
by Palac\'{i}n and Sklinos \cite{Palacin2014} and by Poizat \cite{Poizat_2014}.
Specifically, they both proved, using different methods, that for
any integer $q\geq2$ the structure $(\Z,+,0,\prod_{q})$ is superstable
of $U$-rank $\omega$, where $\prod_{q}=\{q^{n}\,:\, n\in\N\}$.
Palac\'{i}n and Sklinos also showed the same result for other examples,
such as $(\Z,+,0,\mbox{Fac})$, where $\mbox{Fac}=\{n!\,:\, n\in\N\}$.
Conant \cite{Conant_Sparsity_2017} and Lambotte and Point \cite{LambottePoint2017}
independently generalized these results. For a subset $A\subseteq\Z$
with either an upper bound or a lower bound, they give some sparsity
conditions on $A$ which are sufficient for the structure $(\Z,+,0,A)$
to be superstable of $U$-rank $\omega$. Conant also gives sparsity
conditions which are necessary for the structure $(\Z,+,0,A)$ to
be stable.

A different kind of example was given recently by Kaplan and Shelah
in \cite{KaplanShelahPrimes2016}. They proved that for $\mbox{Pr}=\{p\in\Z\,:\,|p|\mbox{ is prime}\}$,
the structure $(\Z,+,0,\mbox{Pr})$ has the independence property
(and even the $n$-independence property for all $n$) hence it is unstable. On the other hand,  assuming Dickson\textquoteright s Conjecture
\footnote{A strong number-theoretic conjecture about primes in arithmetic progressions,
which generalizes Dirichlet\textquoteright s theorem on prime numbers.
}, it is supersimple of $U$-rank $1$.

In contrast to the above, $(\Z,+,0,<)$ remained the only known unstable
dp-minimal expansion of $(\Z,+,0)$. In \cite[Question 5.32]{Aschenbrenner_et_al_II_2013},
Aschenbrenner, Dolich, Haskell, Macpherson, and Starchenko ask $(\star)$
whether every dp-minimal expansion of $(\mathbb{Z},+,0)$ is a reduct
of $(\mathbb{Z},+,0,<)$. In \cite{Aschenbrenner_et_al_I_2015} the
same authors prove that $(\Z,+,0,<)$ has no proper dp-minimal expansions.
This was later strengthened by Dolich and Goodrick, which proved in
\cite{DolichGoodrick2015} that $(\Z,+,0,<)$ has no proper strong
expansions. Together with a result of Conant which we describe below
(Fact~\ref{conant_result_order}), this means that any other unstable dp-minimal
expansion of $(\mathbb{Z},+,0)$, if exists, is not a reduct, nor an expansion of $(\mathbb{Z},+,0,<)$. 

~

In the first part of this paper we introduce a new family of dp-minimal
expansions of $(\mathbb{Z},+,0)$, thus giving a negative answer to
the question $(\star)$ above. More generally, for every $n\in\N\cup\{\omega\}$
we introduce a family of expansions of $(\mathbb{Z},+,0)$ having
dp-rank $n$. For a prime number $p$, let $v_{p}:\mathbb{Z}\to\mathbb{N}\cup\{\infty\}$
be the $p$-adic valuation, namely, $v_{p}(a)=\mbox{sup}\{k\in\N\,:\, p^{k}|a\}$.
Let $\emptyset\neq P\subseteq\mathbb{N}$ be a (possibly infinite)
set of primes, and let $L_{P}$ be the language $\{+,0\}\cup\{|_{p}\,:\, p\in P\}$,
where each $|_{p}$ is a binary relation. We expand $(\mathbb{Z},+,0)$
to an $L_{P}$-structure $\cZ_{P}$ by interpreting $a|_{p}b$ as
$v_{p}(a)\leq v_{p}(b)$ for each $p\in P$. We denote $T_{P}:=Th(\cZ_{P})$.
For convenience, we enumerate $P$ by $P=\set{p_{\alpha}\,:\,\alpha<|P|}$,
and $p$ without a subscript usually denotes some $p\in P$. If $P=\{p\}$
we write $T_{p}$ instead of $T_{\{p\}}$, etc.

We first prove that $T_{P}$ eliminates quantifiers in a natural definitional expansion. Let $L_{P}^{E}=L_{P}\cup\set{-,1\}\cup\{D_{n}\,:\, n\geq1}$
where $-$ and $1$ are interpreted in the obvious way, and for each
$n\geq1$, $D_{n}$ is an unary relation symbol interpreted as $\{na\,:\, a\in\Z\}$. 
\begin{thm}
\label{thm:QE}For every nonempty set $P$ of primes, the theory $T_{P}$
eliminates quantifiers in the language $L_{P}^{E}$.
\end{thm}
After proving this we were informed that a similar result has been
proved independently by François Guignot \cite{Guignot2012}, and
again by Nathanaël Mariaule \cite[Corollary 2.11]{Mariaule}.

~

Using quantifier elimination, we are able to determine the dp-rank
of $T_{P}$.
\begin{thm}
\label{thm:our_result_dp_rank}For every nonempty set $P$ of primes,
$\mbox{dp-rank}(T_{P})=|P|$.
\end{thm}
In particular, for a single prime $p$ we have that $T_p$ is dp-minimal, i.e. $\mbox{dp-rank}(T_{p})=1$.

~

~

We now move to our second result. We first give some context and history.
\begin{defn}
Let $L_{1}$ and $L_{2}$ be two first-order languages, and let $\cM_{1}$
be an $L_{1}$-structure and $\cM_{2}$ an $L_{2}$-structure, both
with the same underlying universe $M$. Let $A\subseteq M$ be a set
of parameters.
\begin{enumerate}
\item We say that $\cM_{1}$ is an \textbf{$A$}-\textbf{reduct} of $\cM_{2}$,
and $\cM_{2}$ is an \textbf{$A$}-\textbf{expansion} of $\cM_{1}$,
if for every $n\geq1$, every subset of $M^{n}$ which is $L_{1}$-definable
over $\emptyset$ (equivalently, over $A$) is also $L_{2}$-definable
over $A$. When $A=M$ we just say that $\cM_{1}$ is a\textbf{ reduct}
of $\cM_{2}$, and $\cM_{2}$ is an\textbf{ expansion} of $\cM_{1}$.
We will mostly use this with either $A=\emptyset$ or $A=M$.
\item We say that $\cM_{1}$ and $\cM_{2}$ are \textbf{$A$}-\textbf{interdefinable}
if $\cM_{1}$ is an \textbf{$A$}-reduct of $\cM_{2}$ and $\cM_{2}$
is an \textbf{$A$}-reduct of $\cM_{1}$. When $A=M$ we just say
that $\cM_{1}$ and $\cM_{2}$ are \textbf{interdefinable. }
\item Let $A\subseteq B\subseteq M$ be another set of parameters. We say
that $\cM_{1}$ is a\textbf{ $B$}-\textbf{proper $A$}-\textbf{reduct}
of $\cM_{2}$, and $\cM_{2}$ is a\textbf{ $B$}-\textbf{proper $A$}-\textbf{expansion}
of $\cM_{1}$, if $\cM_{1}$ is an \textbf{$A$}-reduct of $\cM_{2}$,
but $\cM_{2}$ is not a \textbf{$B$}-reduct of $\cM_{1}$. When $B=M$
we just say \textbf{proper} instead of \textbf{$B$}-proper. We will
mostly use this with either $B=M$ or $B=\emptyset$.
\end{enumerate}
\end{defn}
~

Let $\cM_{1}$ be an $L_{1}$-structure and $\cM_{2}$ an $L_{2}$-structure,
both with the same underlying universe $M$, and suppose that $\cM_{1}$
is a \textbf{$\emptyset$}-reduct of $\cM_{2}$. Then we can replace
$L_{2}$ by $L_{2}\cup L_{1}$, interpreting each $L_{1}$-symbol
in $\cM_{2}$ as it is interpreted in $\cM_{1}$. As we have not added
new $\emptyset$-definable sets, this new structure is \textbf{$\emptyset$}-interdefinable
with the original $\cM_{2}$. Therefore we may always assume for simplicity
of notation that $L_{1}\subseteq L_{2}$ and $\cM_{1}=\cM_{2}|_{L_{1}}$.

~

\textbf{$A$}-reducts are preserved by elementary extensions and elementary
substructures containing $A$, in the following sense:
\begin{obs}
\label{obs:reducts_preserved_by_elementary_substructures}Let $\cM\prec\cN$
be two $L$-structures with universes $M$ and $N$ respectively.
Let $A\subseteq M$ and let $\cN^{\tg}$ be an \textbf{$A$}-reduct
of $\cN$ with language $L^{\tg}$. Let $\cM^{\tg}$ be the structure
obtained by restricting the relations and functions of $\cN^{\tg}$
to $M$. Then: 
\begin{enumerate}
\item \label{enu:reducts_preserved_by_elementary_substructures__elementary_substructure}$\cM^{\tg}$
is well-defined, it is an \textbf{$A$}-reduct of $\cM$, and $\cM^{\tg}\prec\cN^{\tg}$. 
\item \label{enu:A-properness_preserved}$\cN^{\tg}$ is an \textbf{$A$}-proper
\textbf{$A$}-reduct of $\cN$ if and only if $\cM^{\tg}$ is an \textbf{$A$}-proper
\textbf{$A$}-reduct of $\cM$.
\item \label{enu:properness_preserved}$\cN^{\tg}$ is a proper \textbf{$A$}-reduct
of $\cN$ if and only if $\cM^{\tg}$ is a proper \textbf{$A$}-reduct
of $\cM$.
\end{enumerate}
\end{obs}
\begin{rem}
\obsref{reducts_preserved_by_elementary_substructures} is not necessarily
true if $A\not\subseteq M$. If $\cN^{\tg}$ contains a constant $c\notin M$,
or a $n$-ary function $f$ such that $f(M^{n})\not\subseteq M$,
then $\cM^{\tg}$ is not well-defined. Even when it is well-defined,
the rest is still not necessarily true. For example, let $\cM=(\mathbb{Z},+,0,1,<)$,
and let $\cN=(N,+,0,1,<)$ be a nontrivial elementary extension of
$\cM$. Let $b\in N$ be a positive infinite element, and let $\cN^{\tg}=(N,+,0,1,[0,b])$.
Then $\cM^{\tg}=(\mathbb{Z},+,0,1,\N)\not\prec\cN^{\tg}$ (as $[0,b]$
contains an element $x=b$ such that $x\in[0,b]$ but $x+1\notin[0,b]$).
Also, $\cM^{\tg}$ is interdefinable with $\cM$, but we will see
that $\cN^{\tg}$ is a proper reduct of $\cN$.
\end{rem}

\begin{defn}
Let $\mathcal{F}$ be a family of first-order structures, and let
$\cM\in\mathcal{F}$. We say that $\cM$ is \textbf{$A$}-\textbf{minimal}
in $\mathcal{F}$ if there are no \textbf{$A$}-proper \textbf{$A$}-reducts
of $\cM$ in $\mathcal{F}$. We say that $\cM$ is \textbf{$A$}-\textbf{maximal}
in $\mathcal{F}$ if there are no \textbf{$A$}-proper \textbf{$A$}-expansions
of $\cM$ in $\mathcal{F}$. When $A=M$ we just say that $\cM$ is
\textbf{minimal} or \textbf{maximal}, respectively.
\end{defn}
An example of this phenomenon was given by Pillay and Steinhorn,
who proved in \cite{Pillay_1987} that $(\N,<)$ has no proper $o$-minimal
expansions, i.e., it is a maximal $o$-minimal structure. Another
example was given by Marker, who proved in \cite{Marker_1990} that
if $\cN$ is a \textbf{$\emptyset$}-expansion of $(\C,+,\cdot,0,1)$
and a reduct of $(\C,+,\cdot,0,1,\R)$, then $\cN$ is interdefinable
with either $(\C,+,\cdot,0,1)$ or $(\C,+,\cdot,0,1,\R)$, i.e., $(\C,+,\cdot,0,1,\R)$
is minimal among the proper expansions of $(\C,+,\cdot,0,1)$. A much
more recent example, given by Dolich and Goodrick in \cite{DolichGoodrick2015},
was already mentioned above: $(\Z,+,0,<)$ has no proper strong expansions,
i.e., it is maximal among the strong structures\footnote{For a more general example, by Zorn's Lemma, every stable structure
$\cM$ has an expansion which is maximal among the stable expansions
of $\cM$. And as stability is preserved under non-proper expansions,
this maximal expansion may be chosen to be a \textbf{$\emptyset$}-expansion.
Similarly, for every $n\geq1$, by Zorn's Lemma, every stable structure
$\cM$ of dp-rank $n$ has an expansion which is maximal among the
stable expansions of $\cM$ of dp-rank $n$.}.

A concrete example to an even stronger phenomenon was recently given.
Based on a result by Palac\'{i}n and Sklinos \cite{Palacin2014}, Conant
and Pillay proved in \cite{ConantPillay2016} the following:
\begin{fact}[\cite{ConantPillay2016}~Theorem 1.2]
\label{fact:conant_pillay_stable_expansions_finite_dp_rank}$(\Z,+,0,1)$
has no proper stable expansions of finite dp-rank.
\end{fact}
In other words, $(\Z,+,0,1)$ is maximal among the stable structures
of finite dp-rank. This theorem is no longer true if we replace $(\Z,+,0,1)$
by an elementarily equivalent structure $(N,+,0,1)$. Let $(N,+,0,1,|_{p})$
be a nontrivial elementary extension of $(\mathbb{Z},+,0,1,|_{p})$,
let $b\in N$ be such that $\gamma:=v_{p}(b)$ is nonstandard, and let
$B=\{a\in N\,:\, b|_{p}a\}=\{a\in N\,:\, v_{p}(a)\geq\gamma\}$. Then
$(N,+,0,1,B)$ is a proper expansion of $(N,+,0,1)$ of dp-rank $1$,
and in \propref{counterexample_stable_dp_minimal_expansion_of_N_equiv_Z}
we show that it is also stable. 

As $(\mathbb{Z},+,0,<)$ is dp-minimal, an immediate consequence of
the above is that there are no stable structures which are both proper
expansions of $(\Z,+,0)$ and proper reducts of $(\mathbb{Z},+,0,<)$.
In \cite{Conant_No_Intermediate_2016} Conant strengthened this result
by proving that there are no structures \emph{at all} which are both
proper expansions of $(\Z,+,0)$ and proper reducts of $(\mathbb{Z},+,0,<)$.
Again, by interdefinability, we may replace $(\Z,+,0)$ by $(\Z,+,0,1)$
and $(\mathbb{Z},+,0,<)$ by $(\mathbb{Z},+,0,1,<)$. So we have:
\begin{fact}[\cite{Conant_No_Intermediate_2016}~Theorem 1.1]
\label{conant_result_order}$(\mathbb{Z},+,0,1,<)$ is minimal
among the proper expansions of $(\mathbb{Z},+,0,1)$.
\end{fact}
Again, this is no longer true if we replace $(\mathbb{Z},+,0,1,<)$
by an elementarily equivalent structure. In private communication,
Conant mentioned the following possible counterexample: Let $(N,+,0,1,<)$
be a nontrivial elementary extension of $(\mathbb{Z},+,0,1,<)$, let
$b\in N$ be a positive nonstandard element, and let $B=[0,b]$. Then
$(N,+,0,1,B)$ is a proper expansion of $(N,+,0,1)$, and in \propref{counterexample_unstable_order_intermediate}
we show that it is indeed also a proper reduct of $(N,+,0,1,<)$.
Note that the formula $y-x\in B$ defines the ordering on $B$, so
this structure is unstable. We will see (\remref{any_order_intermediate_with_new_one_dimensional_set_is_unstable})
that every structure which is a proper expansion of $(N,+,0,1)$ and
a reduct of $(N,+,0,1,<)$, and which has a definable one-dimensional
set which is not definable in $(N,+,0,1)$, defines a set of the form $[0,b]$ for a positive nonstandard $b$. Hence a stable intermediate structure between $(N,+,0,1,<)$ and $(N,+,0,1)$, if such exists, cannot contain new definable
sets of dimension one.

Nevertheless, a weaker version of Fact~\ref{conant_result_order} does
hold as well for elementarily equivalent structures. As $(\mathbb{Z},+,0,1,<)$
is a \textbf{$\emptyset$}-expansion of $(\mathbb{Z},+,0,1)$, by
Fact~\ref{conant_result_order} it is obviously minimal among the proper
\textbf{$\emptyset$}-expansions of $(\mathbb{Z},+,0,1)$. In $(\mathbb{Z},+,0,1)$,
every element is $\emptyset$-definable, so a proper \textbf{$\emptyset$}-expansion
of $(\mathbb{Z},+,0,1)$ is the same as a \textbf{$\emptyset$}-proper
\textbf{$\emptyset$}-expansion of $(\mathbb{Z},+,0,1)$. Now if $\cN$
is a \textbf{$\emptyset$}-proper \textbf{$\emptyset$}-reduct of
$(\mathbb{Z},+,0,1,<)$, and a \textbf{$\emptyset$}-proper \textbf{$\emptyset$}-expansion
of $(\mathbb{Z},+,0,1)$, then also in $\cN$ every element is $\emptyset$-definable,
so $\cN$ is a proper reduct of $(\mathbb{Z},+,0,1,<)$. Hence $(\mathbb{Z},+,0,1,<)$
is \textbf{$\emptyset$}-minimal among the \textbf{$\emptyset$}-proper
\textbf{$\emptyset$}-expansions of $(\mathbb{Z},+,0,1)$. By \obsref{reducts_preserved_by_elementary_substructures},
we get:
\begin{cor}
\label{cor:conant_result_order_elementary_extension_version}Let $(N,+,0,1,<)$
be an elementary extension of $(\Z,+,0,1,<)$. Then $(N,+,0,1,<)$
is \textbf{$\emptyset$}-minimal among the \textbf{$\emptyset$}-proper
\textbf{$\emptyset$}-expansions of $(N,+,0,1)$.
\end{cor}
Conant's proof of Fact~\ref{conant_result_order} is very elementary
from a model-theoretic point of view. In particular, it does not use
\factref{conant_pillay_stable_expansions_finite_dp_rank}. On the
other hand, it is somewhat complicated, involving detailed analysis
of definable sets in arbitrary dimension. Conant asked whether this
theorem can be proved using model theoretic methods which incorporate
\factref{conant_pillay_stable_expansions_finite_dp_rank}. Here we
give such a proof. Utilizing a basic property of (un)stability, we
were able to prove minimality among unstable expansions by reducing
the problem to the one-dimensional case (in an elementary extension),
which is much easier.

Using the same reduction to dimension $1$, and additional technical
lemmas, we prove:
\begin{thm}
\label{thm:our_result_unstable_p_adic_elementary_extension_version}Let
$(N,+,0,1,|_{p})$ be an elementary extension of $(\Z,+,0,1,|_{p})$.
Then $(N,+,0,1,|_{p})$ is \textbf{$\emptyset$}-minimal among the
\uline{unstable} \textbf{$\emptyset$}-proper \textbf{$\emptyset$}-expansions
of $(N,+,0,1)$.
\end{thm}
Combined with \factref{conant_pillay_stable_expansions_finite_dp_rank}
and \thmref{our_result_dp_rank}, we obtain:
\begin{thm}
\label{thm:our_result_full_p_adic_elementary_extension_version}Let
$(N,+,0,1,|_{p})$ be an elementary extension of $(\Z,+,0,1,|_{p})$.
Then $(N,+,0,1,|_{p})$ is \textbf{$\emptyset$}-minimal among the
\textbf{$\emptyset$}-proper \textbf{$\emptyset$}-expansions of $(N,+,0,1)$.
\end{thm}
In particular:
\begin{cor}
\label{cor:our_result_full_p_adic}$(\mathbb{Z},+,0,1,|_{p})$ is
minimal among the proper expansions of $(\mathbb{Z},+,0,1)$.
\end{cor}
Again, \corref{our_result_full_p_adic} fails for elementary extensions, see  \propref{counterexample_unstable_valuation_intermediate}.

~

\section{Axioms and basic sentences of $T_{P}$.}

In this section, we present a set of axioms for a subtheory $T_{P}^{\tg}\subseteq T_{P}$,
and use them to prove a number of (families of) sentences of $T_{P}^{\tg}$.
In section 3 we will use these sentences to prove quantifier elimination for $T_{P}^{\tg}$,
from which it will also follow that in fact $T_{P}^{\tg}=T_{P}$. 

For convenience, we will work with the valuation functions $v_{p}$ instead of the relations $|_{p}$. Let us
define a multi-sorted language $L_{P}^{M}$ for the valuations $v_p$ on $(\Z,+,0)$ for $p\in P$ as follows: let $Z$ be the main sort with a function symbol $+$ and a constant symbol $0$, interpreted as in $(\mathbb{Z},+,0)$. For each $p\in P$ we add a distinct
sort $\Gamma_{p}$ together with the symbols $<_{p}$, $0_{p}$, $S_{p}$ and $\infty_{p}$, interpreted as a distinct copy of $\left(\mbox{\ensuremath{\N\cup}\{\ensuremath{\infty}\}},<,0,S,\infty\right)$ where $S$ is the successor function. Finally, we add a function symbol
$v_{p}:Z\to\Gamma_{p}$, interpreted as the $p$-adic valuation\footnote{It could be interesting to consider the language with just one
sort $(N,<,0,S,\infty)$ for valuation, instead of one for each $p\in P$. Since different valuations are allowed to interact with each other, the resulting structures might be much more complicated.}. When
confusion is possible, we denote by $\mathbf{v}_{p}$ the usual valuation in the metatheory, to distinguish it from the function symbol $v_{p}$.
We omit the subscript $p$ in $<_{p}$, $0_{p}$, $S_{p}$, $\infty_{p}$ and $\Gamma_p$ when
no confusion is possible. 
~

We use the following standard notation. Let $k\in\N$ be a
nonnegative integer.  
\begin{itemize}
\item In the $Z$ sort, $\underline{k}$ denotes $\underbrace{1+1+\dots+1}_{k\mbox{ times}}$
if $k>0$ and $0$ if $k=0$. Also, $\underline{-k}$ denotes $-\underline{k}$.
\item For an element $a$ from $Z$, $ka$ denotes $\underbrace{a+a+\dots+a}_{k\mbox{ times}}$ if $k>0$ and $0$ if $k=0$, $(-k) a$ denotes $-(ka)$, similarly
for a variable $x$ in place of $a$. 
\item For an element $\gamma$ from
$\Gamma_{p}$, $\gamma+\underline{k}$ denotes $\underbrace{S(S(\dots(\gamma)\dots))}_{k\mbox{ times}}$,
similarly for a variable $u$ in place of $\gamma$, and $\underline{k}$ is an abbreviation for $0+\underline{k}$. 
\end{itemize}
~

The group $(\Z,+,0)$ with valuations $v_p$ for $p\in P$ can be seen as an $L_{P}$-structure and an $L_{P}^{M}$-structure which are interdefinable (with imaginaries) so they essentially define the same sets. We will therefore not distinguish
between the $L_{P}$-structure and the $L_{P}^{M}$-structure on $(\Z,+,0)$, except when dealing with dp-rank,
where we always refer to the one-sorted language $L_P$.

For quantifier elimination we define $L_{P}^{M,E}=L_{P}^{M}\cup\{-,1\}\cup\{D_{n}\,:\, n\geq1\}$
as before. In the $L_{P}^{E}$-structure on $\Z$, every atomic formula without
parameters is definable by a quantifier-free formula without parameters and with variables in the $Z$ sort in the $L_{P}^{M,E}$-structure on $\Z$, and vice-versa.
Hence quantifier elimination in $L_{P}^{E}$ follows from quantifier
elimination in $L_{P}^{M,E}$. We will therefore prove quantifier
elimination for the $L_{P}^{M,E}$-structure on $\Z$. 

For $a\in \Z$ and $p\in P$, let $(a_i)_{i\in \N}$ be the $p$-adic representation of $a$, i.e. $a = \sum_{i\in \N} a_ip^{i}$ and each $a_i$ is in $\set{0,\dots,p-1}$. For $\gamma \in \N$, the \emph{prefix of $a$ of length $\gamma$} is the sequence $( a_i)_{i<\gamma}$. The \emph{ball} of radius $\gamma$ and center $a$ is the set of all integers with same prefix of length $\gamma$ as $a$.

\begin{prop}
\label{prop:the_axioms}The following sentences are true in $\cZ_{P}$
and therefore are in $T_{P}$:
\begin{enumerate}
\item \label{enu:axiom_axioms_of_Z}Any axiomatization for $Th(\mathbb{Z},+,-,0,1,\{D_{n}\}_{n\geq1})$
in the $Z$ sort.
\item \label{enu:axiom_axioms_of_N}For each $p$, any axiomatization of
$Th(\mbox{\ensuremath{\N}}\cup\{\infty\},<,0,S,\infty)$ in the sort
$(\Gamma_{p},<_{p},0_{p},S_{p},\infty_{p})$.
\item \label{enu:axiom_v_non_negative_and_infty_on_0}For each $p$ : $\forall x(v_{p}(x)\geq0\wedge(v_{p}(x)=\infty\leftrightarrow x=0))$.
\item \label{enu:axiom_v_sum_geq_min}For each $p$ : $\forall x,y(v_{p}(x+y)\geq\mbox{min}(v_{p}(x),v_{p}(y)))$.
\item \label{enu:axiom_v_different_vs_sum_eq_min}For each $p$ : $\forall x,y(v_{p}(x)\neq v_{p}(y)\rightarrow v_{p}(x+y)=\mbox{min}(v_{p}(x),v_{p}(y)))$.
\item \label{enu:axiom_v_multiplicative}For each $p$ and $0\neq n\in\Z$
: $\forall x(v_{p}(nx)=v_{p}(x)+\underline{\mathbf{v}_{p}(n)})$.
\item \label{enu:axiom_v_p_eq_1}For each $p$ : $v_{p}(\underline{p})=1$.
\item \label{enu:axiom_exactly_p^k_possible_extensions}For each $p$ and
$k\in\N$ : Every ball in $v_{p}$ of radius $\gamma$ consists of
exactly $p^{k}$ disjoint balls of radius $\gamma+k$.
\end{enumerate}
\end{prop}
\begin{proof}
(\ref{enu:axiom_axioms_of_Z})-(\ref{enu:axiom_v_p_eq_1}) are obvious.
For (\ref{enu:axiom_exactly_p^k_possible_extensions}), let $a\in\mathbb{Z}$
and $\gamma\in\mathbb{N}$. The ball in $v_{p}$ of radius $\gamma$
around $a$ is the set of integers such that, in $p$-adic representation,
their prefix of length $\gamma$ is the same as the prefix of $a$
of length $\gamma$. There are $p$ possibilities for each digit,
so $p^{k}$ possibilities for the $k$ digits with indices $\gamma,\dots,\gamma+k-1$,
which exactly correspond to the balls of radius $\gamma+k$ contained
in the original ball. 
\end{proof}
Let $T_{P}^{\tg}$ be the theory implied by the axioms (\ref{enu:axiom_axioms_of_Z})-(\ref{enu:axiom_exactly_p^k_possible_extensions}).
All of the following propositions are first order, and we prove them
using only $T_{P}^{\tg}$. Let $\mathcal M$ be some fixed model of $T_P'$, with $\mathcal Z$ the $Z$-sort and $\Gamma_p$ the $\Gamma_p$-sort.

\begin{lem}
\label{lem:immidiate_basic_properties}For each $p$:
\begin{enumerate}
\item \label{enu:immidiate_v_diff_geq_min}$\forall x,y(v_{p}(x-y)\geq\mbox{min}(v_{p}(x),v_{p}(y)))$.
\item \label{enu:immidiate_v_p_translated_by_y_is_surjective}$\forall u\forall y\exists x(v_{p}(x-y)=u)$.
In particular, $v_{p}$ is surjective.
\item \label{enu:immidiate_v_p_of_constants}For each $n\neq0$, $v_{p}(\underline{n})=\underline{\mathbf{v}_{p}(n)}$.
\item \label{enu:immidiate_v_geq_k_equiv_D_p^k}For each $k\geq1$ : $\forall x(v_{p}(x)\geq\underline{k}\leftrightarrow D_{p^{k}}(x))$.
\end{enumerate}
\end{lem}

\begin{proof}
We only prove item (\ref{enu:immidiate_v_p_translated_by_y_is_surjective}), the others are easy to check. By Axiom
(\ref{enu:axiom_exactly_p^k_possible_extensions}) with $k=1$, there
are $x_{1},x_{2}$ such that $v_{p}(x_{1}-y)\geq u$, $v_{p}(x_{2}-y)\geq u$,
and $v_{p}(x_{1}-x_{2})<u+\underline{1}$. Hence by (\ref{enu:immidiate_v_diff_geq_min})
above, $u+\underline{1}>v_{p}(x_{1}-x_{2})=v_{p}((x_{1}-y)-(x_{2}-y))\geq\mbox{min}(v_{p}(x_{1}-y),v_{p}(x_{2}-y))\geq u$.
So either $v_{p}(x_{1}-y)=u$ or $v_{p}(x_{2}-y)=u$.
\end{proof}

The following lemmas are left as an exercice.

\begin{lem}
\label{lem:reduction_to_a_single_D_m_with_m_coprime_to_all_p_k}~
\begin{enumerate}
\item \label{enu:reduction to a single D_m}Let $n_{1},\dots,n_{l}\in\N$,
and let $N\in\N$ be such that $n_{i}|N$ for all $1\leq i\leq l$. Let $b_1,\dots, b_n$ be element of $\mathcal Z$.
Then every boolean combination of formulas of the form $D_{n_{i}}(k_{i}x-b_{i})$
is equivalent to a disjunction (possibly empty, i.e. a contradiction)
of formulas of the form $D_{N}(x-\underline{r}_{j})$, where for each
$j$, $r_{j}\in\{0,1,\dots,N-1\}$. 
\item \label{enu:reduction to D_m with m coprime to all p_k}Let $m\in\N$
and let $m^{\tg},k\in\N$ be such that $m=p^{k}\cdot m^{\tg}$ and
$\mbox{gcd}(m^{\tg},p)=1$. Let $r\in\Z$, and let $r_{1}=r\mbox{ mod }m'$,
$r_{2}=r\mbox{ mod }p^{k}$. Then the formula $D_{m}(x-\underline{r})$
is equivalent to $D_{m'}(x-\underline{r}_{1})\wedge(v_{p}(x-\underline{r}_{2})\geq k)$.
\end{enumerate}
\end{lem}

\begin{lem}
\label{lem:reduction_to_inequalities_with_only_a_single_use_of_the_valuation}~
For $a_1$ and $a_2$ in $\mathcal Z$.
\begin{enumerate}
\item For every $k\geq1$, the formula $v_{p}(x-a_{1})<v_{p}(x-a_{2})+\underline{k}$
is equivalent to 
\[
v_{p}(x-a_{2})<v_{p}(a_{2}-a_{1})\vee v_{p}(x-a_{2})>v_{p}(a_{2}-a_{1})\vee v_{p}(x-a_{1})<v_{p}(a_{2}-a_{1})+\underline{k}\mbox{.}
\]

\item For every $k\geq0$, the formula $v_{p}(x-a_{1})+\underline{k}<v_{p}(x-a_{2})$
is equivalent to $v_{p}(x-a_{2})>v_{p}(a_{2}-a_{1})+\underline{k}$.
\end{enumerate}
\end{lem}

\begin{lem}
\label{lem:elimination of irrelevant upper bounds}~
For a fixed $p\in P$, $a_0,a_1$ in $\mathcal Z$ and $\gamma_0,\gamma_1\in \Gamma_p$.
\begin{enumerate}
\item Every formula of the form $v_{p}(x-a_{0})\geq\gamma_{0}\wedge v_{p}(x-a_{1})<\gamma_{1}$
where $\gamma_{0}\geq\gamma_{1}$, is either inconsistent (if $v_{p}(a_{0}-a_{1})\geq\gamma_{1}$)
or equivalent to $v_{p}(x-a_{0})\geq\gamma_{0}$ (if $v_{p}(a_{0}-a_{1})<\gamma_{1}$).\\

\item Every formula of the form $v_{p}(x-a_{0})\geq\gamma_{0}\wedge v_{p}(x-a_{1})<\gamma_{1}$
where $\gamma_{0}<\gamma_{1}$ and $v_{p}(a_{0}-a_{1})<\gamma_{0}$
is equivalent to just $v_{p}(x-a_{0})\geq\gamma_{0}$.
\end{enumerate}
\end{lem}

\begin{lem}
\label{lem:reduction to a single lower bound}Every two balls in $\Gamma_{p}$
are either disjoint, or one is contained in the other. More generally, for $(a_i)_i\in \mathcal Z$, $(\gamma_i)_i\in \Gamma_p$,
every conjunction of formulas of the form $v_{p}(x-a_{i})\geq\gamma_{i}$
is either inconsistent, or equivalent to a single formula $v_{p}(x-a_{i_{0}})\geq\gamma_{i_{0}}$,
where $\gamma_{i_{0}}=max\{\gamma_{i}\}$.\end{lem}

\begin{defn}
\label{def:the order of containment of prefixes}For $a,b\in Z$,
$\gamma,\delta\in\Gamma_{p}$, define $(a,\gamma)\leq_{p}(b,\delta)$
if $\gamma\leq\delta$ and $v_{p}(a-b)\geq\gamma$. \\
Define $(a,\gamma)\sim_{p}(b,\delta)$ if $(a,\gamma)\leq_{p}(b,\delta)$
and $(a,\gamma)\geq_{p}(b,\delta)$.
\end{defn}

$(a,\gamma)\leq_{p}(b,\delta)$ means that $\gamma\leq \delta$ and, in $p$-adic representation,
the prefix of $a$ of length $\gamma$ is contained in the prefix
of $b$ of length $\delta$. This is equivalent to saying that the
ball of radius $\gamma$ around $a$ (namely, $\{x\,:\, v_{p}(x-a)\geq\gamma\}$)
contains the ball of radius $\delta$ around $b$. 

Note that $\leq_{p}$ and $\sim_{p}$ are defined by quantifier-free
formulas, and so do not depend on the model containing the elements
under consideration.

~

\begin{lem}\label{lem:elimination of redundant upper bounds}
The parameters $a_i$ are in $\mathcal Z$ and $\gamma_i$ are in $\Gamma_p$ for some $p\in P$.
\begin{enumerate}
\item Every formula of
the form $v_{p}(x-a_{0})\geq\gamma_{0}\wedge\bigwedge_{m=1}^{n}v_{p}(x-a_{m})<\gamma_{m}$
is equivalent to the formula $v_{p}(x-a_{0})\geq\gamma_{0}\wedge\bigwedge_{m\in C}v_{p}(x-a_{m})<\gamma_{m}$,
for every $C\subseteq\{1,\dots,n\}$ such that $\{(a_{m},\gamma_{m})\,:\, m\in C\}$
contains at least one element from each $\sim_{p}$-equivalence class
of $\leq_{p}$-minimal elements among $\set{(a_m,\gamma_m) : 1\leq m\leq n}$ (i.e. representatives for all the
maximal balls). In particular, this is true for $C$ consisting of
one element from each such class, i.e. for $C$ an antichain.
\item Assume that $(a_0,\gamma_0),\dots,(a_n,\gamma_n)$ are such that  for all $1\leq m\leq n$ we have $\gamma_{m}>\gamma_{0}$,
$v_{p}(a_{m}-a_{0})\geq\gamma_{0}$, and $k_{m}:=\gamma_{m}-\gamma_{0}$
is a standard integer. Assume further that $\{(a_{m},\gamma_{m})\,:\,1\leq m\leq n\}$
is an antichain with respect to $\leq_{p}$. Then every formula of the form $v_{p}(x-a_{0})\geq\gamma_{0}\wedge\bigwedge_{m=1}^{n}v_{p}(x-a_{m})<\gamma_{m}$ is equivalent to a formula
of the form $\bigvee_{i=1}^{l}v_{p}(x-b_{i})\geq\gamma_{N}$ with
$N$ such that $\gamma_N = \mathrm{max}\set{\gamma_{m}\,:\,1\leq m\leq n}$,
where for all $i$, $v_{p}(b_{i}-a_{0})\geq\gamma_{0}$ and for $i\neq j$,
$v_{p}(b_{i}-b_{j})<\gamma_{N}$, and where\\
$l=p^{k_{N}}-\sum_{m}p^{k_{N}-k_{m}}\geq0$ (it may be that $l=0$,
i.e. a contradiction). In particular, $l$ does not depend on the
model $\mathcal M$ of $T_P'$ containing the $a_{i}$'s and $\gamma_{i}$'s. 
\end{enumerate}
\end{lem}
\begin{proof}
We prove \textit{(1)}. Let $C$ be such. For each $1\leq m\leq n$ there is an $m^{\prime}$
such that $(a_{m^{\prime}},\gamma_{m^{\prime}})\leq(a_{m},\gamma_{m})$
and $(a_{m^{\prime}},\gamma_{m^{\prime}})$ is minimal among the $(a_{i},\gamma_{i})$'s.
So $\forall x(v_{p}(x-a_{m^{\prime}})<\gamma_{m^{\prime}}\rightarrow v_{p}(x-a_{m})<\gamma_{m})$.
As $\{(a_{i},\gamma_{i})\,:\, i\in C\}$ contains one element from
each $\sim$-equivalence class of $\leq$-minimal elements, we may
assume $m^{\prime}\in C$. \\
We prove \textit{(2)}. Assume without loss of generality that $\gamma_{1}\leq\gamma_{2}\leq\dots\leq\gamma_{n}$.
Let $b_{0},\dots,b_{p^{k_{n}}-1}$ be the $x_{0},\dots,x_{p^{k}-1}$ from
Axiom \ref{enu:axiom_exactly_p^k_possible_extensions} for $k_{n}$,
$\gamma_{0}$, $a_{0}$. Then $v_{p}(x-a_{0})\geq\gamma_{0}$ is equivalent
to $\bigvee_{i=0}^{p^{k_{n}}-1}(v_{p}(x-b_{i})\geq\gamma_{n})$. For
every $m\geq1$, let $c_{m,0},\dots,c_{m,p^{k_{n}-k_{m}}-1}$ be the
$x_{0},\dots,x_{p^{k}-1}$ from Axiom \ref{enu:axiom_exactly_p^k_possible_extensions}
for $k_{n}-k_{m}$, $\gamma_{m}$, $a_{m}$. Then $v_{p}(x-a_{m})\geq\gamma_{m}$
is equivalent to $\bigvee_{i=0}^{p^{k_{n}-k_{m}}-1}(v_{p}(x-c_{m,i})\geq\gamma_{n})$.
For every $m$, $v_{p}(a_{0}-a_{m})\geq\gamma_{0}$, so for every
$0\leq i\leq p^{k_{n}-k_{m}}-1$, $v_{p}(c_{m,i}-a_{0})\geq\gamma_{0}$.
Hence by the choice of $\{b_{j}\}_{j}$, there is a unique $s_{m,i}<p^{k_n}$
such that $v_{p}(c_{m,i}-b_{s_{m,i}})\geq\gamma_{n}$. So $v_{p}(x-a_{m})\geq\gamma_{m}$
is equivalent to $\bigvee_{i=0}^{p^{k_{n}-k_{m}}-1}(v_{p}(x-b_{s_{m,i}})\geq\gamma_{n})$. 

By the choice of $\{c_{m,i}\}_{i}$, $\bigwedge_{i\neq j}(v_{p}(c_{m,i}-c_{m,j})<\gamma_{n})$,
so also $\bigwedge_{i\neq j}(v_{p}(b_{s_{m,i}}-b_{s_{m,j}})<\gamma_{n})$.
In particular, $i\mapsto s_{m,i}$ is injective for a fixed $m$, hence $F_{m}:=\{s_{m,i}\,:\,0\leq i\leq p^{k_{n}-k_{m}}-1\}$
is of size $p^{k_{n}-k_{m}}$. 

The sets $\{F_{m}\}_{m=1}^{n}$ must be mutually disjoint. Otherwise,
there are $m_{1}<m_{2}$ and $i,j$ such that $s_{m_{1},i}=s_{m_{2},j}$.
Since $v_{p}(c_{m_{1},i}-b_{s_{m_{1},i}})\geq\gamma_{n}$ and $v_{p}(c_{m_{2},j}-b_{s_{m_{2},j}})\geq\gamma_{n}$
we get $v_{p}(c_{m_{1},i}-c_{m_{2},j})\geq\gamma_{n}\geq\gamma_{m_{1}}$.
Since $v_{p}(c_{m_{1},i}-a_{m_{1}})\geq\gamma_{m_{1}}$ and $v_{p}(c_{m_{2},j}-a_{m_{2}})\geq\gamma_{m_{2}}\geq\gamma_{m_{1}}$,
we get $v_{p}(a_{m_{1}}-a_{m_{2}})\geq\gamma_{m_{1}}$, a contradiction
to the antichain assumption. 

Let $F:=\bigcup_{m=1}^{n}F_{m}$. By the above, $\mid F\mid=\sum_{m}p^{k_{n}-k_{m}}$
and 
\[
\forall x(\,(v_{p}(x-a_{0})\geq\gamma_{0}\wedge\bigwedge_{m=1}^{n}v_{p}(x-a_{m})<\gamma_{m})\leftrightarrow({\displaystyle \bigvee_{i\notin F}}v_{p}(x-b_{i})\geq\gamma_{n})\,)\,\,\,)\mbox{.}
\]
\end{proof}

\begin{lem}\label{lem:Preservation of a solution}
For all elements $a_i, a_{i,j}$ in $\mathcal Z$ and $\gamma_i$ in $\Gamma_p$ for some $p\in P$, we have the following.
\begin{enumerate}
\item If b is a solution to $v_{p}(x-a_{0})\geq\gamma_{0}\wedge\bigwedge_{i=1}^{n}v_{p}(x-a_{i})<\gamma_{i}$
and $v_{p}(b'-b)\geq\gamma:=max\{\gamma_{0},\dots,\gamma_{n}\}$ then
$b'$ is also a solution.
\item Every
formula of the form $v_{p}(x-a_{0})\geq\gamma_{0}\wedge\bigwedge_{m=1}^{n}v_{p}(x-a_{m})<\gamma_{m}$
where for each $1\leq m\leq n$, $\gamma_{m}\geq\gamma_{0}+\underline{n}$,
has a solution. 
\item If $p_{1},\dots,p_{l}\in P$
are different primes not dividing $m$ and $\gamma_{i}\in \Gamma_{p_i}$, then every formula of the
form\\
$(\bigwedge_{k=1}^{l}v_{p_{k}}(x-a_{k})\geq\gamma_{k})\wedge D_{m}(x-r)$
has an infinite number of solutions.
\item If
$p_{1},\dots,p_{l}\in P$ are different primes not dividing $m$ and $\gamma_{k,j}\in \Gamma_{p_k}$, then
every formula of the form\\
$$\bigwedge_{k=1}^{l}\left(v_{p_{k}}(x-a_{k,0})\geq\gamma_{k,0}\wedge\bigwedge_{i=1}^{n_{k}}v_{p_{k}}(x-a_{k,i})<\gamma_{k,i}\right)\wedge D_{m}(x-r)$$
where for each $1\leq k\leq l$ and $1\leq i\leq n_{k}$, $\gamma_{k,i}\geq\gamma_{k,0}+\underline{n}_{k}$,
has an infinite number of solutions. In particular, this holds if each
$\gamma_{k,i}-\gamma_{k,0}$ is a nonstandard integer.
\end{enumerate}
\end{lem}

\begin{proof}
The proofs of \textit{(1)} and \textit{(3)} are left as an easy exercice. We prove \textit{(2)}. By Axiom \ref{enu:axiom_exactly_p^k_possible_extensions} for $k=n$, there are $b_{0},\dots,b_{p^{n}-1}$ such that
for all $i$, $v_{p}(b_{i}-a_{0})\geq\gamma_{0}$, and for all $i\neq j$,
$v_{p}(b_{i}-b_{j})<\gamma_{0}+\underline{n}$. Then some $b_{i}$
must satisfy $\bigwedge_{m=1}^{n}v_{p}(x-a_{m})<\gamma_{m}$, otherwise,
since $p^{n}>n$, by the Pigeonhole Principle there are $i\neq j$
and $m$ such that $v_{p}(b_{i}-a_{m})\geq\gamma_{m}$ and $v_{p}(b_{j}-a_{m})\geq\gamma_{m}$,
and therefore also $v_{p}(b_{i}-b_{j})\geq\gamma_{m}\geq\gamma_{0}+\underline{n}$,
a contradiction. \\
We prove \textit{(4)}. For each $1\leq k\leq l$, by \textit{(2)}
the formula $v_{p_{k}}(x-a_{k,0})\geq\gamma_{k,0}\wedge(\bigwedge_{i=1}^{n_{k}}v_{p_{k}}(x-a_{k,i})<\gamma_{k,i})$
has a solution $b_{k}$. Let $\gamma_{k}:=\mathrm{max}\{\gamma_{k,0},\dots,\gamma_{k,n_{k}}\}$.
By \textit{(3)} the formula
$(\bigwedge_{k=1}^{l}v_{p_{k}}(x-b_{k})\geq\gamma_{k})\wedge D_{m}(x-r)$
has an infinite number of solutions $\{b_{j}^{\prime}\}_{j\geq1}$.
By \textit{(1)}, every $b_{j}^{\prime}$
is a solution to
$$\bigwedge_{k=1}^{l}\left(v_{p_{k}}(x-a_{k,0})\geq\gamma_{k,0}\wedge\bigwedge_{i=1}^{n_{k}}v_{p_{k}}(x-a_{k,i})<\gamma_{k,i}\right)\wedge D_{m}(x-r)$$
\end{proof}
~\\

\section{Quantifier elimination}

\begin{proof}[Proof of \thmref{QE}]
As mentioned previously, we will in fact prove quantifier elimination for $T_{P}^{\tg}\subseteq T_{P}$.
It is enough to prove that for all models $\mathcal{M}_{1}$ and $\mathcal{M}_{2}$
of $T_{P}^{\tg}$, with a common substructure $A$, and for all formulas
$\phi(x)$ in a single variable $x$ over $A$ which are a conjunction
of atomic or negated atomic formulas, we have $\mathcal{M}_{1}\vDash\exists x\phi(x)\Rightarrow\mathcal{M}_{2}\vDash\exists x\phi(x)$.
Let $\mathcal{M}_{1}$,~$\mathcal{M}_{2}$, $A$ and $\phi(x)$ be
such, and let $b\in\mathcal{M}_{1}$ be such that $\mathcal{M}_{1}\vDash\phi(b)$. 

~

As $v_p$ is surjective for all $p\in P$, we may assume that $x$ is of the $Z$ sort. Since $\phi$
contains only finitely many symbols from $L_{P}$, we may assume for
simplicity of notation that $P$ is finite. So $\phi(x)$ is equivalent%
\footnote{The negation of a formula of form (5) is
$v_{p_{\alpha}}(n_{i,1}x-a_{i,1})\geq v_{p_{\alpha}}(n_{i,2}x-a_{i,2})+\underline{k}$,
which is equivalent to $v_{p_{\alpha}}(n_{i,2}x-a_{i,2})+\underline{k-1}<v_{p_{\alpha}}(n_{i,1}x-a_{i,1})$
if $k>0$, which is of form (6), and to $v_{p_{\alpha}}(n_{i,2}x-a_{i,2})<v_{p_{\alpha}}(n_{i,1}x-a_{i,1})+1$
if $k=0$, which is of form (5). Similarly for the negation of a formula
of form (6). Also, (7) and (8) are in essence special cases of (5) or (6), but
they are required because in $A$ the valuation may be not surjective.} to a conjunction of formulas of the forms:
\begin{enumerate}
\item $n_{i}x=a_{i}$, for some $n_{i}\neq0$.
\item $n_{i}x\neq a_{i}$, for some $n_{i}\neq0$.
\item $D_{m_{i}}(n_{i}x-a_{i})$, for some $n_{i}\neq0$.
\item $\neg D_{m_{i}}(n_{i}x-a_{i})$, for some $n_{i}\neq0$.
\item $v_{p_{\alpha}}(n_{i,1}x-a_{i,1})<v_{p_{\alpha}}(n_{i,2}x-a_{i,2})+\underline{k}_{i}$, for some $p_\alpha\in P$, $n_{i,1}\neq0$ or $n_{i,2}\neq0$, and $k_i\in \N$.
\item $v_{p_{\alpha}}(n_{i,1}x-a_{i,1})+\underline{k}_{i}<v_{p_{\alpha}}(n_{i,2}x-a_{i,2})$, for some $p_\alpha\in P$, $n_{i,1}\neq0$ or $n_{i,2}\neq0$, and $k_i\in \N$.
\item $v_{p_{\alpha}}(n_{i}x-a_{i})\geq\gamma_{i}$, for some $p_\alpha\in P$ and $n_{i}\neq0$.
\item $v_{p_{\alpha}}(n_{i}x-a_{i})<\gamma_{i}$, for some $p_\alpha\in P$ and $n_{i}\neq0$.
\end{enumerate}
~

By multiplicativity of the valuations we may assume that for all formulas
of forms (5) or (6), either $n_{i,1}=n_{i,2}$, $n_{i,1}=0$ or $n_{i,2}=0$.
Therefore, by Lemma~\ref{lem:reduction_to_inequalities_with_only_a_single_use_of_the_valuation},
we may assume that every formula of form (5) or (6) is equivalent
to a formula of form (7) or (8).

~

By Lemma~\ref{lem:reduction_to_a_single_D_m_with_m_coprime_to_all_p_k},
the conjunction of all the formulas of the forms (3) or (4) is equivalent
to a formula of the form
$$\bigvee_{j}\left(D_{m_{j}}(x-r_{j})\wedge\bigwedge_{\alpha<|P|}v_{p_{\alpha}}(x-s_{j,\alpha})\geq\underline{k}_{j,\alpha}\right)$$
where for all $j$ and $\alpha$, $\mathrm{gcd}(m_{j},p_{\alpha})=1$. As $\mathcal{M}_{1}\vDash\phi(b)$, this disjunction is not empty. Let $D_{m}(x-r)\wedge\bigwedge_{\alpha<|P|}v_{p_{\alpha}}(x-s_{\alpha})\geq\underline{k}_{\alpha}$
be one of the disjuncts which are satisfied by $b$. It is enough
to find $b'\in\mathcal{M}_{2}$ which satisfies this disjunct, along
with all the formulas of other forms. Note that $v_{p_{\alpha}}(x-s_{\alpha})\geq\underline{k}_{\alpha}$
is of form (7), so altogether we want to find $b'\in\mathcal{M}_{2}$
which satisfies a conjunction of formulas of the forms:
\begin{enumerate}
\item $n_{i}x=a_{i}$, $n_{i}\neq0$.
\item $n_{i}x\neq a_{i}$, $n_{i}\neq0$.
\item $D_{m}(x-r)$, where for all $\alpha<|P|$, $\mbox{gcd}(m,p_{\alpha})=1$ (only a single such formula).
\item $v_{p_{\alpha}}(n_{i}x-a_{i})\geq\gamma_{i}$, $\alpha<|P|$, $n_{i}\neq0$.
\item $v_{p_{\alpha}}(n_{i}x-a_{i})<\gamma_{i}$, $\alpha<|P|$, $n_{i}\neq0$.
\end{enumerate}
~

By a standard argument, we may assume that the conjunction does not contain formulas of form (1).
For each formula of form (2), there is at most one element which does
not satisfy it. So it is enough to prove that there are infinitely
many elements in $\mathcal{M}_{2}$ which satisfy all the formulas
of forms (3), (4) or (5). 

Let $n:=\prod_{i}n_{i}$. By multiplicativity of the valuations, the
conjunction of formulas of forms (3), (4) or (5) is equivalent to the conjunction of: 
\begin{enumerate}
\item $v_{p_{\alpha}}(nx-\frac{n}{n_{i}}a_{i})\geq\gamma_{i}+\mathbf{v}_{p_{\alpha}}(\frac{n}{n_{i}})$.
\item $v_{p_{\alpha}}(nx-\frac{n}{n_{i}}a_{i})<\gamma_{i}+\mathbf{v}_{p_{\alpha}}(\frac{n}{n_{i}})$.
\item $D_{nm}(nx-nr)$.
\end{enumerate}
By substituting $y=nx$, it is equivalent to satisfy:
\begin{enumerate}
\item $v_{p_{\alpha}}(y-\frac{n}{n_{i}}a_{i})\geq\gamma_{i}+\mathbf{v}_{p_{\alpha}}(\frac{n}{n_{i}})$.
\item $v_{p_{\alpha}}(y-\frac{n}{n_{i}}a_{i})<\gamma_{i}+\mathbf{v}_{p_{\alpha}}(\frac{n}{n_{i}})$.
\item $D_{nm}(y-nr)$.
\item $D_{n}(y)$.
\end{enumerate}
Notice that formula (4) is already implied by formula (3). Again by
Lemma~\ref{lem:reduction_to_a_single_D_m_with_m_coprime_to_all_p_k}, we
may exchange $D_{nm}(y-nr)$ by a formula $D_{m'}(y-r')$, where for
all $\alpha<|P|$, $\mbox{gcd}(m^{\tg},p_{\alpha})=1$. Also, by Lemma~\ref{lem:reduction to a single lower bound}
we may assume that for each $\alpha<|P|$, there is only one formula
of form (1). Altogether, it is enough to prove that in $\mathcal{M}_{2}$
there are infinitely many elements which satisfy the conjunction of
the following formulas: 
\begin{enumerate}
\item \label{enu:before elimination of irrelevant upper bounds}$v_{p_{\alpha}}(x-a_{\alpha,0})\geq\gamma_{\alpha,0}$
for all $\alpha<|P|$.
\item $v_{p_{\alpha}}(x-a_{\alpha,i})<\gamma_{\alpha,i}$ for all $\alpha<|P|$,
$1\leq i\leq n_{\alpha}$.\hspace*{\fill}$\circledast$
\item $D_{m}(x-r)$, where for all $\alpha<|P|$, $\mbox{gcd}(m,p_{\alpha})=1$
(only a single such formula).
\end{enumerate}
~

By Lemma~\ref{lem:elimination of irrelevant upper bounds} (and since this
formula is consistent in $\mathcal{M}_{1}$) we may assume that for
all $\alpha<|P|$, $1\leq i\leq n_{\alpha}$ we have $\gamma_{\alpha,0}<\gamma_{\alpha,i}$
and $v_{p_{\alpha}}(a_{\alpha,0}-a_{\alpha,i})\geq\gamma_{\alpha,0}$.
By Lemma~\ref{lem:elimination of redundant upper bounds} \textit{(1)}, we may assume
that for each $\alpha<|P|$, the set 
$$\{(a_{\alpha,i},\gamma_{\alpha,i})\,:\,1\leq i\leq n_{\alpha}\,,\,\gamma_{\alpha,i}-\gamma_{\alpha,0}\mbox{ is a standard integer}\}$$
is an antichain with respect to $\leq_{p_{\alpha}}$ (Definition~\ref{def:the order of containment of prefixes}). 

For each $\alpha<|P|$, let $S_{\alpha}=\{0\leq i\leq n_{\alpha}\,:\,\gamma_{\alpha,i}-\gamma_{\alpha,0}\mbox{ is a standard integer}\}$
and $\gamma_{\alpha,0}^{\prime}=max\{\gamma_{\alpha,i}\,:\, i\in S_{\alpha}\}$.
For $s=1,2$ and for each $\alpha<|P|$, by  Lemma~\ref{lem:elimination of redundant upper bounds} \textit{(2)}
the conjunction $v_{p_{\alpha}}(x-a_{\alpha,0})\geq\gamma_{\alpha,0}\wedge\bigwedge_{i\in S_{\alpha}}v_{p_{\alpha}}(x-a_{\alpha,i})<\gamma_{\alpha,i}$
is equivalent in $\mathcal{M}_{s}$ to a formula of the form $\bigvee_{i=1}^{l_{\alpha}}v_{p_{\alpha}}(x-a_{\alpha,0,i}^{s})\geq\gamma_{\alpha,0}^{\prime}$,
where for all $i$, $a_{\alpha,0,i}^{s}\in\mathcal{M}_{s}$ and $l_{\alpha}$ does not depend on $s$. Note that $a_{\alpha,0,i}^{s}$ may not be in $A$. Furthermore, by Lemma~\ref{lem:elimination of redundant upper bounds} \textit{(2)}, $v_{p_{\alpha}}(a_{\alpha,0,i}^{s}-a_{\alpha,0})\geq\gamma_{\alpha,0}$
and for $i\neq j$, $v_{p_{\alpha}}(a_{\alpha,0,i}^{s}-a_{\alpha,0,j}^{s})<\gamma_{\alpha,0}^{\prime}$. 

Together, the conjunction of the formulas in \hyperref[enu:before elimination of irrelevant upper bounds]{$\circledast$}
is equivalent in $\mathcal{M}_{s}$ to the disjunction $\psi_{s}=\bigvee_{k=1}^{l}\psi_{s,k}$,
where for each $k$, $\psi_{s,k}$ is the conjunction of the following
formulas: 
\begin{enumerate}
\item $v_{p_{\alpha}}(x-a_{\alpha,0,k}^{s})\geq\gamma_{\alpha,0}^{\prime}$
, for all $\alpha<|P|$.
\item $v_{p_{\alpha}}(x-a_{\alpha,i})<\gamma_{\alpha,i}$, for all $\alpha<|P|$,
$i\notin S_{\alpha}$ (so $\gamma_{\alpha,0}<\gamma_{\alpha,i}$ and
$\gamma_{\alpha,i}-\gamma_{\alpha,0}$ is not a standard integer).
\item $D_{m}(x-r)$, where for all $\alpha<|P|$, $\mathrm{gcd}(m,p_{\alpha})=1$
(only a single such formula).
\end{enumerate}
Furthermore, $l=\prod_{\alpha<|P|}l_{\alpha}$ does not depend on $s$.

Since $\psi_{1}$ is consistent in $\mathcal{M}_{1}$ (satisfied by
$nb$), the disjunction for $s=1$ is not empty, i.e., $l\geq1$.
And since $l$ does not depend on $s$, the disjunction for $s=2$
is also not empty. Consider one such disjunct, $\psi_{2,k}$. By Lemma~\ref{lem:Preservation of a solution} \textit{(4)},
it has an infinite number of solutions. This completes the proof.
\end{proof}

\begin{cor}
$T_{P}^{\tg}$ is a complete theory. Hence $T_{P}^{\tg}=T_{P}$.
\end{cor}

\begin{proof}
By quantifier elimination, it is enough to show that $T_{P}^{\tg}$ decides every atomic sentence. These are just the sentences equivalent to one of the forms: $\underline{n}_{1}=\underline{n}_{2}$ in any sort, $\underline{k}_{1}<_{p}\underline{k}_{2}$
in $\Gamma_{p}$, $D_{m}(\underline{n})$ in the $Z$ sort and $v_{p}(\underline{n}_{1})<v_{p}(\underline{n}_{2})$ in the $Z$ sort,
all of which are clearly decided by $T_{P}^{\tg}$. \end{proof}

\begin{rem}\label{rem:QE_one_variable_simpler_form}
Suppose $\mathcal M\models T_P$ and $\phi(x)$ is a consistent formula in a single variable with parameters from $\mathcal M$. Then by quantifier elimination and Lemmas \ref{lem:reduction_to_a_single_D_m_with_m_coprime_to_all_p_k}
and \ref{lem:reduction_to_inequalities_with_only_a_single_use_of_the_valuation}, $\phi(x)$ is equivalent to a disjunction of formulas, which are either of the form $x = a$ or of the form
$$D_m(x-r)\wedge \bigwedge_j nx\neq a_j \wedge \bigwedge_{p\in F}\left( v_p(n_px - a_{p,0})\geq \gamma_{p,0} \wedge \bigwedge_{i=1}^{l_p} v_p(n_px - a_{p,i})< \gamma_{p,i}\right),$$
where $F\subseteq P$ is finite and $\mathrm{gcd}(m,p) = 1$ for all $p\in F$. Moreover, one may assume $\mathrm{gcd}(n_p,p)=1$ for each $p\in F$.
\end{rem}

For $p$ a single prime number and $\mathcal{M}\models T_p$, the following lemma says that the definable subgroups of $(\mathcal M,+)$
are only those of the form $m\mathcal M\cap \set{a\in \mathcal M : v(a)\geq \gamma}$, for $m\in \Z$ and $\gamma\in \Gamma$ and for each such
defining formula, there are only finitely many possible $m$'s when
varying the parameters of the formula. 
\begin{lem}
\label{lem:Characterizing_definable_subgroups}
For a single prime $p$, let $\phi(x,y)$ be
any $L_p^M$-formula, and let $\theta(y)$ be the formula for ``$(\phi(x,y)\,,\,+)$
is a subgroup''. Then there are $n_{1},\dots,n_{k}\geq1$, having $\mbox{gcd}(n_{i},p)=1$
for each $i$, such that the following sentence is true in $T_p$:
\[
\forall y\left(\theta(y)\to\bigvee_{i=1}^{k}\exists w\forall x(\phi(x,y)\leftrightarrow (D_{n_{i}}(x)\wedge(v_{p}(x)\geq v_{p}(w)))\right)\mbox{.}
\]
\end{lem}
\begin{proof}
It is enough to work in $\Z$.
By quantifier elimination (and \lemref{reduction_to_a_single_D_m_with_m_coprime_to_all_p_k}
(\ref{enu:reduction to D_m with m coprime to all p_k})), $\phi(x,y)$
is equivalent to a formula of the form $\bigvee_{i}\bigwedge_{j}\phi_{i,j}(x,y)$,
where for each $i,j$, $\phi_{i,j}(x,y)$ is one of the following:
\begin{enumerate}
\item $t_{i,j}(x,y)=0$, where $t_{i,j}(x,y)$ is a $\{+,-,1\}$-term, i.e.,
of the form $k_{i,j}x+l_{i,j}y+\underline{r}_{i,j}$ for $k_{i,j},l_{i,j},r_{i,j}\in\Z$.
\item $t_{i,j}(x,y)\neq0$, where $t_{i,j}(x,y)$ is a $\{+,-,1\}$-term.
\item $v(t_{i,j}(x,y))\geq v(s_{i,j}(x,y))$, where $t_{i,j}(x,y)$, $s_{i,j}(x,y)$
are $\{+,-,1\}$-terms. ($v(t_{i,j}(x,y))<v(s_{i,j}(x,y))$ is equivalent
to $v(p\cdot t_{i,j}(x,y))\leq v(s_{i,j}(x,y))$, which is of the
same form).
\item $D_{m_{i,j}}(t_{i,j}(x,y))$, where $t_{i,j}(x,y)$ is a $\{+,-,1\}$-term
and $\mbox{gcd}(m_{i,j},p)=1$. 
\end{enumerate}
For each $i$, let $J_{i}=\{j\,:\,\phi_{i,j}(x,y)\mbox{ is of the form }D_{m_{i,j}}(t_{i,j}(x,y))\}$,
and let $m_{i}=\prod_{j\in J_{i}}m_{i,j}$. As in the proof of \lemref{reduction_to_a_single_D_m_with_m_coprime_to_all_p_k}
(\ref{enu:reduction to a single D_m}), the satisfaction of the formula
$D_{m_{i,j}}(t_{i,j}(x,y))$ depends only on the reminders of $x$
and $y$ mod $m_{i,j}$, which are determined by the reminders of
$x$ and $y$ mod $m_{i}$. So there is a set $R_{i}\subseteq\{0,1,\dots,m_{i}-1\}^{2}$
such that $\bigwedge_{j\in J_{i}}\phi_{i,j}(x,y)$ is equivalent to
$\bigvee_{(r,s)\in R_{i}}(D_{m_{i}}(x-\underline{r})\wedge D_{m_{i}}(y-\underline{s}))$.
Therefore, $\phi(x,y)$ is equivalent to a formula of the form $\bigvee_{i}(D_{m_{i}}(x-\underline{r}_{i})\wedge D_{m_{i}}(y-\underline{s}_{i})\wedge\bigwedge_{j}\phi_{i,j}(x,y))$,
where $\mbox{gcd}(m_{i},p)=1$ and for each $i,j$, $\phi_{i,j}(x,y)$
is one of the following:
\begin{enumerate}
\item $t_{i,j}(x,y)=0$, where $t_{i,j}(x,y)$ is a $\{+,-,1\}$-term.
\item $t_{i,j}(x,y)\neq0$, where $t_{i,j}(x,y)$ is a $\{+,-,1\}$-term.
\item $v(t_{i,j}(x,y))\geq v(s_{i,j}(x,y))$, where $t_{i,j}(x,y)$, $s_{i,j}(x,y)$
are $\{+,-,1\}$-terms. 
\end{enumerate}
For each $i$, let $\phi_{i}(x,y)$ be the $i$'th disjunct, i.e.,
the formula $D_{m_{i}}(x-\underline{r}_{i})\wedge D_{m_{i}}(y-\underline{s}_{i})\wedge\bigwedge_{j}\phi_{i,j}(x,y)$.

Let $b\in\Z$ be such that $\phi(\Z,b)$ is a subgroup. If $\phi(\Z,b)$
is finite, it must be $\{0\}$. To account for this case, we may take
$n_{1}=1$, and for $w=0$ we have that $\phi(x,b)$ is equivalent
to $D_{n_{1}}(x)\wedge(v_{p}(x)\geq v_{p}(0))$. If $\phi(\Z,b)$
is infinite, then $\phi(\Z,b)=n\Z$ for some $n\geq1$. Moreover,
there must be an $i_{0}$ such that $\phi_{i_{0}}(\Z,b)$ is infinite.
So $D_{m_{i_{0}}}(b-\underline{s}_{i_{0}})$ holds, hence $\phi_{i_{0}}(x,b)$ is equivalent to just $D_{m_{i_{0}}}(x-\underline{r}_{i_{0}})\wedge\bigwedge_{j}\phi_{i_{0},j}(x,b)$.
As $\phi(\Z,b)$ is infinite, it is clear that no formula $\phi_{i_{0},j}(x,y)$ is of the form (1), hence $\phi_{i_{0}}(x,b)$ is equivalent to $D_{m_{i_{0}}}(x-\underline{r}_{i_{0}})\wedge\bigwedge_{j}\phi_{i_{0},j}(x,b)$,
where for each $j$, $\phi_{i_{0},j}(x,b)$ is one of the following: 
\begin{enumerate}
\item $k_{i_{0},j}x\neq c_{i_{0},j}$.
\item $v(k_{i_{0},j}^{\tg}x-c_{i_{0},j}^{\tg})\geq v(k_{i_{0},j}^{\tg\tg}x-c_{i_{0},j}^{\tg\tg})$.
\end{enumerate}

Applying \lemref{reduction_to_inequalities_with_only_a_single_use_of_the_valuation} to formulas as in (2), we may assume that $\phi_{i_0}(x,b)$ is equivalent to $D_{m_{i_{0}}}(x-\underline{r}_{i_{0}})\wedge\bigwedge_{j}\phi_{i_{0},j}(x,b)$,
where for each $j$, $\phi_{i_{0},j}(x,b)$ is one of the following: 
\begin{enumerate}
\item $k_{i_{0},j}x\neq c_{i_{0},j}$.
\item $v(k_{i_{0},j}x-c_{i_{0},j})\geq\gamma_{i_{0},j}$.
\item $v(k_{i_{0},j}x-c_{i_{0},j})<\gamma_{i_{0},j}$.
\end{enumerate}
The formula $v(k_{i_{0},j}x-c_{i_{0},j})\geq\gamma_{i_{0},j}$ defines
a coset of $p^{\gamma_{i_{0},j}}\Z$,
and the formula $v(k_{i_{0},j}x-c_{i_{0},j})<\gamma_{i_{0},j}$ defines
a finite union of cosets of $p^{\gamma_{i_{0},j}}\Z$.
Let $J=\{j\,:\,\phi_{i_{0},j}(x,b)\mbox{ is of form 2 or 3}\}$, and
let $\delta=\mbox{max}\{\gamma_{i_{0},j}\,:\, j\in J\}$. Then for
every $j\in J$, every coset of $p^{\gamma_{i_{0},j}}\Z$ is a finite
union of cosets of $p^{\delta}\Z$. So $\bigcap_{j\in J}\phi_{i_{0},j}(\Z,b)$
is a finite intersection of finite unions of cosets of $p^{\delta}\Z$,
and hence is itself just a finite union of cosets of $p^{\delta}\Z$
(since every two cosets are either equal or disjoint). Therefore,
$\phi_{i_{0}}(\Z,b)$ is a set of the form $U\backslash F$, where
$F$ is a finite set (the set of points excluded by the inequalities
$k_{i_{0},j}x\neq c_{i_{0},j}$), and $U$ is a finite union of the
form $\bigcup_{j=1}^{N}((m_{i_{0}}\Z+r_{i_{0}})\cap(p^{\delta}\mathbb{Z}+c_{j}))$.
For each $j$, $(m_{i_{0}}\Z+r_{i_{0}})\cap(p^{\delta}\mathbb{Z}+c_{j})$
is a coset of $m_{i_{0}}p^{\delta}\mathbb{Z}$ (it is not empty, since
$\mbox{gcd}(m_{i_{0}},p)=1$), so $U$ is of the form $\bigcup_{j=1}^{N}(m_{i_{0}}p^{\delta}\mathbb{Z}+d_{j})$.
As $\phi_{i_{0}}(\Z,b)$ is infinite, this union is not empty.

Now, $(m_{i_{0}}p^{\delta}\mathbb{Z}+d_{1})\backslash F\subseteq U\backslash F=\phi_{i_{0}}(\Z,b)\subseteq\phi(\Z,b)=n\Z$,
so $n$ divides $m_{i_{0}}p^{\delta}$ since $F$ is finite. Write $n=n^{\tg}p^{\gamma}$
with $\mbox{gcd}(n^{\tg},p)=1$. Then $n^{\tg}|m_{i_{0}}$, and in
particular, $n^{\tg}\leq m_{i_{0}}$. So $\phi(x,b)$ is equivalent
to $D_{n}(x)$, which is equivalent to $D_{n^{\tg}}(x)\wedge v(x)\geq\gamma$,
and $n^{\tg}\leq m_{i_{0}}$. Recall that $i_{0}$ depends on $b$,
but there are only finitely many $i$'s, so $m=\mbox{max}\{m_{i}\}$
exists, and hence, for any $b$ such that $\phi(x,b)$ is a subgroup,
there is an $n^{\tg}\leq m$ with $\mbox{gcd}(n^{\tg},p)=1$, and
there is a $\gamma$ such that $\phi(x,b)$ is equivalent to $D_{n^{\tg}}(x)\wedge v(x)\geq\gamma$,
and we are done.
\end{proof}
~

\section{dp-rank of $T_{P}$}

Quantifier elimination now enables us to determine the dp-rank of
$T_{P}$. We first review two equivalent definitions of dp-rank. More
details about dp-rank can be found, e.g. in \cite{Simon_2015}. We work in a monster model $\mathbb M$ of some complete $L$-theory $T$, for some langage $L$.
\begin{defn}
Let $\phi(x,b)$ be an $L$-formula, with parameters $b$ from $\mathbb M$, and let $\kappa$ be a
(finite or infinite) cardinal. We say $\mbox{dp-rank}(\phi(x,b))<\kappa$
if for every family $(I_{t}\,:\, t<\kappa)$ of mutually indiscernible
sequences over $b$ and $a\models \phi(x,b)$, there is $t<\kappa$ such that
$I_{t}$ is indiscernible over $ab$.

~

We say that $\mbox{dp-rank}(\phi(x,b))=\kappa$ if $\mbox{dp-rank}(\phi(x,b))<\kappa^{+}$
but not $\mbox{dp-rank}(\phi(x,b))<\kappa$. We say that $\mbox{dp-rank}(\phi(x,b))\leq \kappa$ if $\mbox{dp-rank}(\phi(x,b))<\kappa$ or $\mbox{dp-rank}(\phi(x,b))=\kappa$. Note that if $\kappa$ is a
limit cardinal, it may happen that $\mbox{dp-rank}(\phi(x,b))<\kappa$ but
$\mbox{dp-rank}(\phi(x,b))\geq\lambda$ for all $\lambda<\kappa$.

For a theory $T$ we denote $\mbox{dp-rank}(T)=\mbox{dp-rank}(x=x)$
where $|x|=1$. If $\mbox{dp-rank}(T)=1$ we say that $T$ is \emph{dp-minimal}.
\end{defn}

\begin{defn}
Let $\kappa$ be
a cardinal. An \emph{ict-pattern of length $\kappa$} consists
of:
\begin{itemize}
\item a collection of formulas $(\phi_{\alpha}(x;y_\alpha)\,:\,\alpha<\kappa)$, with $\abs{x} = 1$,
\item an array $(b_{i}^{\alpha}\,:\, i<\omega,\,\alpha<\kappa)$ of tuples,
with $|b_{i}^{\alpha}|=|y_{\alpha}|$
\end{itemize}
such that for every $\eta:\kappa\to\omega$ there exists an element $a_{\eta}\in \mathbb M$
such that 
\[
\models\phi_{\alpha}(a_\eta;b_i^\alpha)\iff\eta(\alpha)=i.
\]

We define $\kappa_{ict}$ as the minimal $\kappa$ such that
there does not exist an ict-pattern of length $\kappa$.\end{defn}
\begin{fact}[{\cite[Proposition 4.22]{Simon_2015}}]
For any cardinal $\kappa$,
we have $\mbox{dp-rank}(T)<\kappa$ if and only if $\kappa_{ict}\leq\kappa$. \end{fact}
\begin{prop}
\label{prop:dp_minimality_of_one_valuation}For any prime $p$, $T_{p}$
is dp-minimal (in the one-sorted language).\end{prop}
\begin{proof}
Denote $L=L_{p}^{E}$ and $T=T_{p}$. Let $L^{-}$ contain the
symbols of $L$, except for the divisibility relations $\{D_{n}\}_{n\geq1}$.
Let $\mathcal{Z}^{-}$ be the reduct of $\mathcal{Z}_{p}$ to $L^{-}$.
Let $\mathbb{Q}_{p}^{-}$ be $\mathbb{Q}_{p}$ as an $L^{-}$-structure.
It is a reduct of the structure $(\mathbb{Q}_{p},+,-,\cdot,0,1,|_{p})$,
which is dp-minimal (see \cite[Theorem 6.6]{Dolich2009}), and therefore
is also dp-minimal. Note that $\mathcal{Z}^{-}$ is a substructure
of $\mathbb{Q}_{p}^{-}$.

Let $L^{\prime}=L\cup\{Z\}$. Interpret $Z$ in $\mathbb{Q}_{p}$
as $\mathbb{Z}$, and interpret each $D_{n}$ such that $D_{n}\cap\mathbb{Z}$
is the usual divisibility relation and $D_{n}\cap(\mathbb{Q}_{p}\backslash\mathbb{Z})=\emptyset$,
thus making it an $L^{\prime}$-structure $\mathbb{Q}_{p}^{\prime}$.
Let $\mathcal{M}$ be an $\omega_{1}$-saturated model of $Th(\mathbb{Q}_{p}^{\prime})$,
and let $A=Z(\mathcal{M})$ be the interpretation of $Z$ in it. Then
$A$ is an $\omega_{1}$-saturated model of $T$.

Suppose that $T$ is not dp-minimal. Then there are formulas $\phi(x,y)$,
$\psi(x,z)$ in $L$ with $|x|=1$, and elements $(b_{i}\,:\, i<\omega)$,
$(c_{j}\,:\, j<\omega)$, $(a_{i,j}\,:\, i,j<\omega)$ in $A$ such
that $\phi(a_{i,j},b_{i^{\prime}})$ if and only if $i=i^{\prime}$ and $\psi(a_{i,j},c_{j^{\prime}})$
iff $j=j^{\prime}$. By \thmref{QE} we may assume that $\phi$,~$\psi$
are quantifier-free and in disjunctive normal form. Let $N$ be the
largest $n$ such that $D_{n}$ appears in $\phi$ or $\psi$. Color
each pair $(i,j)$ such that $i>j$ by $a_{i,j}\mbox{ mod }N!$. By
Ramsey Theorem, we may assume that all the elements $a_{i,j}$ with
$i>j$ have the same residue modulo $N!$, and so modulo all $n\leq N$. 

Write $\phi$ as $\bigvee_{k}\bigwedge_{l}(\phi_{k,l}^{\prime}\wedge\phi_{k,l}^{\prime\prime})$
and $\psi$ as $\bigvee_{k}\bigwedge_{l}(\psi_{k,l}^{\prime}\wedge\psi_{k,l}^{\prime\prime})$,
where $\phi_{k,l}^{\prime}$, $\psi_{k,l}^{\prime}$ are atomic or negated atomic $L^{-}$-formulas
and $\phi_{k,l}^{\prime\prime}$, $\psi_{k,l}^{\prime\prime}$ are
atomic or negated atomic formulas containing no relations other than $\{D_{n}\}_{n\geq1}$.
For each $k$, denote by $\phi_{k}$, $\psi_{k}$ the formulas $\bigwedge_{l}(\phi_{k,l}^{\prime}\wedge\phi_{k,l}^{\prime\prime})$
and $\bigwedge_{l}(\psi_{k,l}^{\prime}\wedge\psi_{k,l}^{\prime\prime})$
respectively.

For every $i>j$ we have $\phi(a_{i,j},b_{i})$, so there is a $k_{i,j}$
such that $\phi_{k_{i,j}}(a_{i,j},b_{i})$. Again by Ramsey Theorem,
we may assume that all the $k_{i,j}$'s are equal to some $k_{0}$,
so for every $i>j$ we have $\phi_{k_{0}}(a_{i,j},b_{i})$. For every
$i^{\prime}\neq i$ we have $\neg\phi(a_{i^{\prime},j},b_{i})$, so
in particular $\neg\phi_{k_{0}}(a_{i^{\prime},j},b_{i})$. Similarly,
we may assume that for some $k_{1}$, for every $i>j$ we have $\psi_{k_{1}}(a_{i,j^{\prime}},c_{j})$
iff $j=j^{\prime}$. 

Let $\phi_{k}^{\prime}$, $\psi_{k}^{\prime}$ be the formulas obtained
from $\phi_{k}$, $\psi_{k}$ respectively, by deleting all the formulas
$\phi_{k,l}^{\prime\prime}$, $\psi_{k,l}^{\prime\prime}$. So $\phi_{k}^{\prime}$,
$\psi_{k}^{\prime}$ are $L^{-}$-formulas. 

For every $m\in\mathbb{N}$, let $I_{m}=\{m+1,\dots,2m\}$, $J_{m}=\{1,\dots,m\}$.
For every $(i,j)\in I_{m}\times J_{m}$, we have $\phi_{k_{0}}(a_{i,j},b_{i})$
and therefore also $\phi_{k_{0}}^{\prime}(a_{i,j},b_{i})$. Let $i\neq i^{\prime}\in I_{m}$,
and suppose for a contradiction that $\phi_{k_{0}}^{\prime}(a_{i^{\prime},j},b_{i})$,
i.e. $\bigwedge_{l}(\phi_{k_{0},l}^{\prime}(a_{i^{\prime},j},b_{i}))$.
But we know that $\neg\phi_{k_{0}}(a_{i^{\prime},j},b_{i})$, so for
some $l_{0}$ we have $\neg\phi_{k_{0},l_{0}}^{\prime}(a_{i^{\prime},j},b_{i})\vee\neg\phi_{k_{0},l_{0}}^{\prime\prime}(a_{i^{\prime},j},b_{i})$.
Therefore, we get $\neg\phi_{k_{0},l_{0}}^{\prime\prime}(a_{i^{\prime},j},b_{i})$.
But from $\phi_{k_{0}}(a_{i,j},b_{i})$ we also get $\phi_{k_{0},l_{0}}^{\prime\prime}(a_{i,j},b_{i})$.
Together, this contradicts the fact that all the elements $a_{i,j}$
with $i>j$ have the same residue modulo all $n\leq N$.

Altogether, in $A$, for every $(i,j)\in I_{m}\times J_{m}$ we have
$\phi_{k_{0}}^{\prime}(a_{i,j},b_{i^{\prime}})$ if and only if $i=i^{\prime}$,
and similarly also $\psi_{k_{1}}^{\prime}(a_{i,j},c_{j^{\prime}})$
iff $j=j^{\prime}$. Since $\phi_{k_{0}}^{\prime}$, $\psi_{k_{1}}^{\prime}$
are quantifier-free, and $A$ is a substructure of $\mathcal{M}$,
this holds also in $\mathcal{M}$. As $m$ is arbitrary, this contradicts
the dp-minimality of $Th(\mathbb{Q}_{p}^{-})$.\end{proof}
\begin{lem}
\label{lem:subadditivity_of_dp-rank_with_respect_to_unions_of_non-interacting_languages}Let
$L=\bigcup_{\alpha<\kappa}L_{\alpha}$ be a language such that every
atomic formula in $L$ is in $L_{\alpha}$ for some $\alpha$.
Let T be an $L$-theory that eliminates quantifiers, and for $\alpha<\kappa$
let $T_{\alpha}$ be its reduction to $L_{\alpha}$. Let $\mu_{\alpha}$ be cardinals such that $\mbox{dp-rank}(T_{\alpha})\leq\mu_{\alpha}$. Then $\mbox{dp-rank}(T)\leq\sum_{\alpha<\kappa}\mu_{\alpha}$, where $\sum$ is the cardinal sum.\end{lem}
\begin{proof}
Suppose not. Let $\lambda:=\sum_{\alpha<k}\mu_{\alpha}$. Then there
is a family $(\mathcal{I}_{t}\,:\, t<\lambda^{+})$
of mutually indiscernible sequences over $\emptyset$, $\mathcal{I}_{t}=(a_{t,i}\,:\, i\in I_{t})$,
and a singleton $b$, such that for all $t$, $\mathcal{I}_{t}$ is
not indiscernible over $b$. For every $t<\lambda^{+}$, let $\phi_{t}(\bar{x})=\phi_{t}(\bar{x},b)$
be a formula over $b$ and let $\bar c_{t,1}$ and $\bar c_{t,2}$ be two
finite tuples of elements of $\mathcal I _t$ of length $|\bar{x}|$ such that $\phi_{t}(\bar{c}_{t,1})$
and $\neg\phi_{t}(\bar{c}_{t,2})$, i.e. witnessing the non-indiscernibility
of $\mathcal{I}_{t}$ over $b$. By quantifier elimination in $T$,
we may assume that $\phi_{t}$ is quantifier-free. Hence there must
be an atomic formula $\psi_{t}(\bar{x})=\psi_{t}(\bar{x},b)$ such that $\psi_{t}(\bar{c}_{t,1})$ and $\neg\psi_{t}(\bar{c}_{t,2})$.
By the assumption on $L$, there is an $\alpha_{t}<\kappa$ such that
$\psi_{t}(\bar{x},y)$ is in $L_{\alpha_{t}}$. Therefore, there must
be an $\alpha<\kappa$ such that $|\{t<\lambda^{+}\,:\,\alpha_{t}=\alpha\}|>\mu_{\alpha}$,
as otherwise we get $$\lambda^{+}=\left| \bigcup_{\alpha<\kappa}\set{t<\lambda^{+}\,:\,\alpha_{t}=\alpha}\right|\leq\sum_{\alpha<\kappa}\left|\set{t<\lambda^{+}\,:\,\alpha_{t}=\alpha}\right| \leq\sum_{\alpha<\kappa}\mu_{\alpha}=\lambda,$$
a contradiction. But then $(\mathcal{I}_{t}\,:\, t<\lambda^{+},\,\alpha_{t}=\alpha)$
is a family of more than $\mu_{\alpha}$ mutually indiscernible sequences
over $\emptyset$ with respect to $L_{\alpha}$, and for all
$t$ such that $\alpha_t = \alpha$, $\mathcal{I}_{t}$ is not indiscernible over $b$ with respect
to $L_{\alpha}$, a contradiction to $\mbox{dp-rank}(T_{\alpha})\leq\mu_{\alpha}$. 
\end{proof}
Now \thmref{our_result_dp_rank} follows: 
\begin{proof}[Proof of \thmref{our_result_dp_rank}]
$\mbox{dp-rank}(T_{P})\leq|P|$ follows from \propref{dp_minimality_of_one_valuation}
and \lemref{subadditivity_of_dp-rank_with_respect_to_unions_of_non-interacting_languages}
for $L_{P}^{E}=\bigcup_{\alpha<|P|}L_{p_{\alpha}}^{E}$. For $\alpha<|P|$
let $\phi_{\alpha}(x,y)$ be the formula $x|_{p_{\alpha}}y\wedge y|_{p_{\alpha}}x$
(i.e. $v_{p_{\alpha}}(x)=v_{p_{\alpha}}(y)$), and for $\alpha<|P|$,
$i\in\mathbb{N}$ let $a_{\alpha,i}$ be such that $v_{p_{\alpha}}(a_{\alpha,i})=i$.
Let $F\subseteq|P|$ be finite. By Lemma~\ref{lem:Preservation of a solution} \textit{(4)},
for every $\eta:F\to\mathbb{N}$ there is a $b_{\eta}$ such that
for every $\alpha\in F$, $v_{p_{\alpha}}(b_{\eta})=v_{p_{\alpha}}(a_{\alpha,\eta(\alpha)})$.
If $P$ is finite, just take $F=|P|$. Otherwise, by compactness,
there are such $b_{\eta}$ for $F=|P|$ as well. These $\phi_{\alpha}(x,y)$,
$a_{\alpha,i}$ and $b_{\eta}$ form an ict-pattern of length $|P|$,
so $\mbox{dp-rank}(T_{P})\geq|P|$.
\end{proof}
~

\section{There are no intermediate structures between $(\mathbb{Z},+,0)$
And $(\mathbb{Z},+,0,|_{p})$}

In this section we focus on a single valuation. Let $p$ be any prime.
Unless stated otherwise, we work in a monster model $\mathcal{M}=(M,+,0,|_{p})$
of $T_{p}$, and denote its value set by $\Gamma$. We may omit the
subscript $p$ when it is clear from the context. Recall that $\Gamma$
is an elementary extension of $(\mathbb{N},<,0,S)$. 

~

\subsection{Preliminaries\label{sub:Cheese-Feng-Shui}}

For $a\in M$, $\gamma\in\Gamma$, we denote by $B(a,\gamma)$ the
definable set $\{x\,:\, v(x-a)\geq\gamma\}$ and call it the \emph{ball}
of \emph{radius} $\gamma$ around $a$. If $\gamma=\infty$ then $B(a,\gamma)$
is just $\{a\}$, and we call such balls \emph{trivial}. Unless stated
otherwise, balls are assumed to be nontrivial. Of course, $a\in B(a,\gamma)$,
and if $b\in B(a,\gamma)$ then $B(b,\gamma)=B(a,\gamma)$. Also,
by \lemref{immidiate_basic_properties} (\ref{enu:immidiate_v_p_translated_by_y_is_surjective}),
if $\delta\neq\gamma$ then $B(a,\delta)\neq B(a,\gamma)$. So the
radius of a ball is well defined. We denote the radius of a ball $B$
by $rad(B)$.

We call a \emph{swiss cheese} any
non-empty set $F$ that can be written as $F=B_{0}\backslash\bigcup_{i=1}^{n}B_{i}$,
where $\{B_{i}\}_{i=0}^{n}$ are balls. Note that this representation
is not unique. As the intersection of any two balls is either empty
or equals one of them, we may always assume that $\{B_{i}\}_{i=1}^{n}$
are nonempty, pairwise disjoint and contained in $B_{0}$. 
\begin{rem}
\label{rem:A-finite-number-of-proper-holes-cannot-cover-a-ball}Rephrasing
Lemma~\ref{lem:Preservation of a solution} \textit{(2)},
if $B_{0}$, $B_{1},\dots,B_{n}$ are balls such that for all $i\geq1$,
$rad(B_{i})\geq rad(B_{0})+\underline{n}$, then $B_{0}\backslash\bigcup_{i=1}^{n}B_{i}\neq\emptyset$.
In particular, this holds if $|rad(B_{i})-rad(B_{0})|\notin\N$.\end{rem}
\begin{prop}
\label{prop:bottom_ball_well_defined}Let $\emptyset\neq F=B_{0}\backslash\bigcup_{i=1}^{n}B_{i}$
be a swiss cheese. Then there exists a unique ball $B_{0}^{\tg}$
such that $F\subseteq B_{0}^{\tg}$ and $B_{0}^{\tg}$ is minimal
with respect to this property. This $B_{0}^{\tg}$ satisfies $B_{0}^{\tg}\subseteq B_{0}$,
$|rad(B_{0}^{\tg})-rad(B_{0})|\in\N$, and it is also the unique ball
$B\subseteq B_{0}$ such that there are at least two distinct balls
$B_{1}^{\tg\tg}$ and $B_2''$, satisfying $rad(B_{j}^{\tg\tg})=rad(B_{0}^{\tg})+1$
and $B_{j}^{\tg\tg}\cap F\neq\emptyset$ for $j=1,2$.
\end{prop}

\begin{proof}
Let $I_{1}=\{1\leq i\leq n\,:\,|rad(B_{i})-rad(B_{0})|\in\N\}$, $I_{2}=\{1,\dots ,n\}\backslash I_{1}$.
By applying Lemma~\ref{lem:elimination of redundant upper bounds} \textit{(2)}
to $B_{0}\backslash\bigcup_{i\in I_{1}}B_{i}\neq\emptyset$, we see
that $B_{0}\backslash\bigcup_{i\in I_{1}}B_{i}=\bigsqcup_{j=1}^{l}B_{j}^{\tg\tg}$,
where $l\geq1$ and for all $j$, $B_{j}^{\tg\tg}\subseteq B_{0}$
and $rad(B_{j}^{\tg\tg})=\mbox{max}\{rad(B_{i})\,:\, i\in I_{1}\}$.
So $F=\bigsqcup_{j=1}^{l}(B_{j}^{\tg\tg}\backslash\bigcup_{i\in I_{2}}B_{i})$.
By \remref{A-finite-number-of-proper-holes-cannot-cover-a-ball},
for each $j$, $B_{j}^{\tg\tg}\backslash\bigcup_{i\in I_{2}}B_{i}\neq\emptyset$.
If $C$ is a ball such that $F\subseteq C$, then for each $j$, $B_{j}^{\tg\tg}\backslash\bigcup_{i\in I_{2}}B_{i}\subseteq C$,
and we claim that in fact $B_{j}^{\tg\tg}\subseteq C$. Indeed, by
Axiom \ref{enu:axiom_exactly_p^k_possible_extensions}, $B_{j}^{\tg\tg}=\bigsqcup_{t=1}^{p}B_{j,t}^{\tg\tg}$
with $rad(B_{j,t}^{\tg\tg})=rad(B_{j}^{\tg\tg})+1$, and again by
\remref{A-finite-number-of-proper-holes-cannot-cover-a-ball}, for
each $t$, $B_{j,t}^{\tg\tg}\backslash\bigcup_{i\in I_{2}}B_{i}\neq\emptyset$.
So $C\cap B_{j,t}^{\tg\tg}\neq\emptyset$ but $C\not\subseteq B_{j,t}^{\tg\tg}$
(as also for $s\neq t$, $C\cap B_{j,s}^{\tg\tg}\neq\emptyset$),
therefore $B_{j,t}^{\tg\tg}\subseteq C$. This holds for all $t$,
hence $B_{j}^{\tg\tg}\subseteq C$. In particular, $B_{1}^{\tg\tg}\subseteq C$.
As $|rad(B_{1}^{\tg\tg})-rad(B_{0})|\in\N$, there are only finitely
many balls $B$ such that $B_{1}^{\tg\tg}\subseteq B\subseteq B_{0}$,
so we may choose $B_{0}^{\tg}$ to be a minimal one (with respect
to inclusion) among those that also satisfy $F\subseteq B$ (exists,
since $B_{0}$ satisfies this). By this choice, $B_{0}^{\tg}\subseteq B_{0}$
and $|rad(B_{0}^{\tg})-rad(B_{0})|\in\N$. If $B$ is another ball
such that $F\subseteq B$, then $F\subseteq B\cap B_{0}^{\tg}$, and
$B\cap B_{0}^{\tg}\neq\emptyset$ is also a ball. Also, as we have
shown, $B_{1}^{\tg\tg}\subseteq B$, so $B_{1}^{\tg\tg}\subseteq B\cap B_{0}^{\tg}\subseteq B_{0}$.
Hence by the choice of $B_{0}^{\tg}$, $B_{0}^{\tg}=B\cap B_{0}^{\tg}\subseteq B$.
This shows that $B_{0}^{\tg}$ is the unique minimal ball containing
$F$. Finally, let $D$ be a ball and assume $F\subseteq D$. By Axiom \ref{enu:axiom_exactly_p^k_possible_extensions}
write $D=\bigsqcup_{t=1}^{p}D_{t}^{\tg\tg}$ with $rad(D_{t}^{\tg\tg})=rad(D)+1$.
Then $D$ is minimal if and only if for all $t$, $F\not\subseteq D_{t}^{\tg\tg}$,
iff there are $t\neq s$ such that $F\cap D_{t}^{\tg\tg}\neq\emptyset$
and $F\cap D_{s}^{\tg\tg}\neq\emptyset$. 
\end{proof}
Let $F$ be a swiss cheese. By \propref{bottom_ball_well_defined}
we may write $F=B_{0}\backslash\bigcup_{i=1}^{n}B_{i}$ where $B_{0}$
is the unique minimal ball containing $F$. We may also assume that
$\{B_{i}\}_{i=1}^{n}$ are nonempty, pairwise disjoint and contained
in $B_{0}$. Unless stated otherwise, all representations are assumed
to satisfy these conditions. We call $B_{0}$ the \emph{outer ball}
of $F$, and define the \emph{radius} of $F$ to be $rad(F):=rad(B_{0})$.
We also call $\{B_{i}\}_{i=1}^{n}$ the \emph{holes} of $F$. Note
that this representation is still not unique (unless there are no
holes at all), as each hole may always be split into $p$ smaller
holes, and sometimes there are sets of $p$ holes which may each be
combined into a single hole. There is a canonical representation for
$F$, namely, the one with the minimal number of holes. But we will
not use it. Rather, when dealing with holes without mentioning a specific
representation, either the intended representation is clear from
the context (e.g., when using \remref{splitting_swiss_cheeses} (\ref{enu:splitting_general})
or (\ref{enu:splitting_to_proper}) to split a swiss cheese with a given
representation), or we may choose any representation and stick with
it.

We say that $B_{i}$ is a \emph{proper hole} of $F$ if $|rad(B_{i})-rad(B_{0})|\notin\N$.
We call $F$ a \emph{proper cheese} if all of its holes are proper.
Note that by \remref{A-finite-number-of-proper-holes-cannot-cover-a-ball},
being a proper cheese does not depend on the representation of the
holes.\\

~

\begin{rem}
\label{rem:splitting_swiss_cheeses}~
\begin{enumerate}
\item \label{enu:a_proper_cheese_is_always_written_nontrivially}If $B_{0}$,$B_{1},\dots,B_{n}$
are balls such that for all $i\geq1$, $B_{i}\subseteq B_{0}$ and
$|rad(B_{i})-rad(B_{0})|\notin\N$, then $B_{0}$ is the outer ball
of the swiss cheese $F=B_{0}\backslash\bigcup_{i=1}^{n}B_{i}$, which is
therefore proper. 
\item \label{enu:splitting_general}Let $F$ be a swiss cheese, and let
$k\geq1$. Then $F$ may be written as a disjoint union $F=\bigsqcup_{i=1}^{l}F_{i}$,
where $1\leq l\leq p^{k}$, and for each $i$, $F_{i}$ is a swiss
cheese such that $rad(F_{i})\geq rad(F)+k$ and $|rad(F_{i})-rad(F)|\in\N$.
Each hole of $F_{i}$ is already a hole of $F$, and each hole of
$F$ is a hole of at most one of the $\{F_{i}\}_{i}$. \\
If $F$ is proper, then $l=p^{k}$ and each $F_{i}$ is a proper cheese
of radius $rad(F_{i})=rad(F)+k$. In this case, each hole of $F$
is a hole of exactly one of the $\{F_{i}\}_{i}$.
\item \label{enu:splitting_to_proper}Let $F=B_{0}\backslash\bigcup_{i=1}^{n}B_{i}$
be a swiss cheese, let $I_{1}=\{1\leq i\leq n\,:\,|rad(B_{i})-rad(B_{0})|\in\N\}$,
and let $k_{0}=\mbox{max}\{rad(B_{i})-rad(B_{0})\,:\, i\in I_{1}\}\in\N$.
Then for each $k\geq k_{0}$, $F$ may be written as a disjoint union
$F=\bigsqcup_{i=1}^{l}F_{i}$, where $1\leq l\leq p^{k}$, and for
each $i$, $F_{i}$ is a \textbf{proper} swiss cheese of radius $rad(F_{i})=rad(F)+k$.
Each hole of $F_{i}$ is already a proper hole of $F$, and each proper
hole of $F$ is a hole of exactly one of the $\{F_{i}\}_{i}$.
\item Let $F^{\tg}$,$F^{\tg\tg}$ be two swiss cheeses of radiuses $\gamma^{\tg}$,$\gamma^{\tg\tg}$
respectively, and let $\gamma=\mbox{max}\{\gamma^{\tg},\gamma^{\tg\tg}\}$.
Then $F^{\tg}\cap F^{\tg\tg}$ is either empty, or also a swiss cheese
of radius $rad(F^{\tg}\cap F^{\tg\tg})\geq\gamma$ such that $|rad(F^{\tg}\cap F^{\tg\tg})-\gamma|\in\N$.
\item If both $F^{\tg}$,$F^{\tg\tg}$ are proper and $\gamma^{\tg}=\gamma^{\tg\tg}$,
and if $F^{\tg}\cap F^{\tg\tg}$ is not empty, then $F^{\tg}$,$F^{\tg\tg}$
have the same outer ball, and $F^{\tg}\cap F^{\tg\tg}$ is also a
proper cheese of the same outer ball. 
\end{enumerate}
\end{rem}

\begin{lem}\label{lem:union_of_cheeses_with_nonempty_intersection_is_also_a_cheese}
Let $F$,$F^{\tg}$ be two swiss cheeses of radiuses $\gamma \leq \gamma ^{\tg}$
respectively.
If $F \cap F ^{\tg}\neq\emptyset$, then $F \cup F ^{\tg}$
is also a swiss cheese, of radius exactly $\gamma$. The set of holes
of $F \cup F ^{\tg}$ is a subset of the union of the set of
holes of $F $ and the set of holes of $F ^{\tg}$. \end{lem}
\begin{proof}
 Write $F =B_{0} \backslash\bigcup_{i=1}^{n}B_{i} $,
$F ^{\tg}=B_{0} ^{\tg}\backslash\bigcup_{j=1}^{m}B_{j} ^{\tg}$. If
$F \cap F ^{\tg}\neq\emptyset$ then $B_{0} \cap B_{0} ^{\tg}\neq\emptyset$,
hence $B_{0} \supseteq B_{0} ^{\tg}$. Therefore, 
\[
F ^{\tg}\backslash F =F ^{\tg}\backslash\left(B_{0} \backslash\bigcup_{i=1}^{n}B_{i} \right)=F ^{\tg}\backslash B_{0} \cup\left(F ^{\tg}\cap\bigcup_{i=1}^{n}B_{i} \right)=\bigcup_{i=1}^{n}F ^{\tg}\cap B_{i} \mbox{.}
\]
For each $i$: if $B_{0} ^{\tg}\cap B_{i} =\emptyset$ then $F ^{\tg}\cap B_{i} =\emptyset$.
Otherwise, as $B_{0} \supseteq B_{0} ^{\tg}$, we also get $B_{i} \subseteq B_{0} ^{\tg}$
($B_{i} \supseteq B_{0} ^{\tg}$ is impossible, as it implies
$F \cap F ^{\tg}=\emptyset$), and in this case, $F ^{\tg}\cap B_{i} =B_{i} \backslash\bigcup_{j=1}^{m}(B_{i} \cap B_{j} ^{\tg})$.
Together, we get 
\[
F \cup F ^{\tg}=F \cup(F ^{\tg}\backslash F )=B_{0} \backslash\left(\bigcup_{i\in I_1}
B_{i} \cup \bigcup_{i\in I_2} \bigcup_{j=1}^{m}(B_{i} \cap B_{j} ^{\tg})\right)
\]
where $I_1$ is the set of $i$ such that $B_0'\cap B_i = \emptyset$ and $I_2$ is the set of $i$ such that $B_i\subseteq B_0'$.
This is a swiss cheese, and as $F \subseteq F \cup F ^{\tg}\subseteq B_{0} $
and $rad(F )=rad(B_{0} )=\gamma $, also $rad(F \cup F ^{\tg})=\gamma$
and $B_{0} $ is its outer ball. For each $i$ such that $B_{i} \subseteq B_{0} ^{\tg}$
and each $j$, either $B_{i} \cap B_{j} ^{\tg}=\emptyset$ (in
which case $B_{i} \cap B_{j} ^{\tg}$ does not appear as a hole of $F\cup F'$), or $B_{i} \cap B_{j} ^{\tg}=B_{i} $
or $B_{i} \cap B_{j} ^{\tg}=B_{j} ^{\tg}$, so the last part
holds.
\end{proof}
~

Sometimes we want disjoint swiss cheeses to also have disjoint outer
balls, but unfortunately, that is not always possible. An example
for this is a union of two swiss cheeses, $F_{1}\cup F_{2}$, with $F_{2}\subseteq B$
where $B$ is one of the holes of $F_{1}$. If $|rad(B)-rad(F_{1})|\in\N$,
we may rewrite $F_{1}$ as a union of swiss cheeses of radius $rad(B)$,
and, together with $F_{2}$, we have a union of swiss cheeses with disjoint
outer balls. But if $|rad(B)-rad(F_{1})|\notin\N$, we cannot
do such a thing. 

\begin{defn}
A \emph{pseudo swiss cheese} is a definable set $P$ such that there
is a swiss cheese $F$ with outer ball $B$ such that $F\subseteq P\subseteq B$.
By the following remark, we may call $B$ the \emph{outer ball} of
$P$, and define the \emph{radius} of $P$ to be $rad(P):=rad(B)$.
We also call $P$ \emph{pseudo proper cheese} if there is a proper cheese $F$ with
outer ball $B$ such that $F\subseteq P\subseteq B$.
 \end{defn}
 

\begin{rem}
\begin{enumerate}
\item In the previous definition, $B$ is uniquely determined by $P$. Indeed, suppose $F_{1}$,$F_{2}$ are two swiss cheeses with outer balls
$B_{1}$,$B_{2}$ respectively, such that $F_{1}\subseteq P\subseteq B_{1}$
and $F_{2}\subseteq P\subseteq B_{2}$. Then $rad(B_{1})=rad(F_{1})\geq rad(B_{2})$
and $rad(B_{2})=rad(F_{2})\geq rad(B_{1})$, so $rad(B_{1})=rad(B_{2})$.
Also, $P\subseteq B_{1}\cap B_{2}\neq\emptyset$, so we must have
$B_{1}=B_{2}$.
\item For every $k\geq1$,
every proper pseudo swiss cheese of radius $\gamma$ can be written
as a union of exactly $p^{k}$ proper pseudo cheeses with disjoint
outer balls of radius exactly $\gamma+k$. 
\item Note that the analogue to \remref{splitting_swiss_cheeses} (\ref{enu:splitting_general})
is not true for pseudo swiss cheeses. For example, let $B$ be a ball of
radius $\gamma$, let $\{B_{i}\}_{i=0}^{p-1}$ be all the balls of
radius $\gamma+1$ contained in $B$, let $\{B_{i,j}\}_{j=0}^{p-1}$
be all the balls of radius $\gamma+2$ contained in $B_{i}$, and
let $C \subseteq B_{0,1}$ be a ball of radius $\delta>\gamma$
such that $|\delta-\gamma|\notin\N$. Then $P=C\sqcup\bigsqcup_{i=0}^{p-1}B_{i,0}$
is a pseudo swiss cheese of radius $\gamma$, but cannot be written as $\leq p$
pseudo swiss cheeses of radius $\geq\gamma+1$, because $P\cap B_{0}$ is
not a pseudo swiss cheese. 
Also note that the intersection of two pseudo swiss cheeses is not necessarily
a single pseudo swiss cheese. For example, take $P\cap B_{0}$ from above.
\end{enumerate}
\end{rem}

\begin{lem}
\label{lem:about_unions_of_pseudo_cheeses}~
\begin{enumerate}
\item \label{enu:union_of_pseudo_cheeses_with_intersecting_balls_is_also_a_pseudo_cheese}Let
$P_{1}$,$P_{2}$ be two pseudo swiss cheeses with outer balls $B_{1}$,$B_{2}$
respectively, such that $rad(B_{1})\geq rad(B_{2})$. If $B_{1}\cap B_{2}\neq\emptyset$
then $P_{1}\cup P_{2}$ is also a pseudo swiss cheese, with outer
ball $B_{2}$. If $P_{2}$ is proper, then $P_{1}\cup P_{2}$ is also
proper.
\item \label{enu:writing_unions_of_pseudo_cheeses_to_be_with_disjoint_balls}Any
finite union of pseudo swiss cheeses may be written as a union of
pseudo swiss cheeses having disjoint outer balls. Also, any finite
union of pseudo proper cheeses may be written as a union of
pseudo proper cheeses having disjoint outer balls. 
\end{enumerate}
\end{lem}
\begin{proof}
~
 We prove \textit{(\ref{enu:union_of_pseudo_cheeses_with_intersecting_balls_is_also_a_pseudo_cheese})}. $B_{1}\cap B_{2}\neq\emptyset$ and $rad(B_{1})\geq rad(B_{2})$,
so $B_{1}\subseteq B_{2}$ and therefore also $P_{1}\subseteq B_{2}$.
Let $F_{2}$ be a swiss cheese with outer ball $B_{2}$ such that
$F_{2}\subseteq P_{2}\subseteq B_{2}$. Then $F_{2}\subseteq P_{1}\cup P_{2}\subseteq B_{2}$.
If $P_{2}$ is proper, then we may take $F_{2}$ to be proper, and
so $P_{1}\cup P_{2}$ is also proper.\\
We prove \textit{(\ref{enu:writing_unions_of_pseudo_cheeses_to_be_with_disjoint_balls})}. Let $A=\bigcup_{i=1}^{n}P_{i}$ such that for each $i$, $P_{i}$
is a pseudo swiss cheese with outer ball $B_{i}$. Let $\{B_{j}^{\tg}\}_{j=1}^{m}$
be the set of all the maximal balls (with respect to inclusion) among
$\{B_{i}\}_{i=1}^{n}$. Then $\{B_{j}^{\tg}\}_{j=1}^{m}$ are pairwise
disjoint. For each $1\leq j\leq m$, let $I_{j}=\{i\,:\, B_{i}\cap B_{j}^{\tg}\neq\emptyset\}$
and $P_{j}^{\tg}=\bigcup_{i\in I_{j}}P_{i}$. So $\{1,\dots,n\}=\bigsqcup_{j=1}^{m}I_{j}$
and therefore $A=\bigcup_{j=1}^{m}P_{j}^{\tg}$. By \textit{(\ref{enu:union_of_pseudo_cheeses_with_intersecting_balls_is_also_a_pseudo_cheese})},
$P_{j}^{\tg}$ is a pseudo swiss cheese with outer ball $B_{j}^{\tg}$.
If for each $i$, $P_{i}$ is proper, then by \textit{(\ref{enu:union_of_pseudo_cheeses_with_intersecting_balls_is_also_a_pseudo_cheese})},
for each $j$, $P_{j}^{\tg}$ is also proper.
\end{proof}

\begin{rem}
\label{rem:D_m_with_m_coprime_to_p_are_dense}The valuation $v_{p}$
induces a topology on $\mathcal{M}$, generated by the balls. By Lemma \ref{lem:Preservation of a solution} \textit{(3)},
if $\mbox{gcd}(m,p)=1$, then the sets defined by $D_{m}(x-r)$ are
dense in $\mathcal{M}$. 
\end{rem}

\begin{lem}
\label{lem:reduction_single_pseudo_cheese_to_ball}Let $P$ be a pseudo
swiss cheese with outer ball $B$ and radius $\alpha$, and assume
$0\in B$. Let $G$ be a dense subgroup of $\mathcal{M}$, and let
$A=P\cap G$. Then there exists $N\in\N$ and $a_{1},\dots,a_{N}\in B\cap G$
such that $\bigcup_{i=1}^{N}(A+a_{i})=B\cap G$.\end{lem}
\begin{proof}
Observe that $B$ is a subgroup of $\mathcal M$ since $0\in B$. Let $F$ be a swiss cheese with outer ball $B$ such that $F\subseteq P\subseteq B$.
By \remref{splitting_swiss_cheeses} (\ref{enu:splitting_to_proper}),
for some finite $k$ we may find a proper cheese $F^{\tg}\subseteq F$
of radius $\alpha+k$. Let $s$ be the number of holes in $F^{\tg}$.
By \remref{splitting_swiss_cheeses} (\ref{enu:splitting_general}),
we may write $F^{\tg}$ as a union of exactly $p^{s}$ proper cheeses
of radius $\alpha+k+s$. As $p^{s}>s$, at least one of these
proper cheeses must have no holes, i.e., must be a ball, say $D$. Let $x\in D$ and $D_0 = D-x$. Then $D_0$
is a subgroup of $B$ of index $N:=p^{k+s}$. Let $x_{1},\dots,x_{N}$
be representatives of the cosets, so $B=\bigcup_{i=1}^{N}x_{i}+D_0$.
For each $i$, let $a_{i}\in x_{i}+D_0\cap G$. As $a_{i}\in B\cap G$
and $A\subseteq B\cap G$, we have $(A+a_{i})\subseteq B\cap G$, and
therefore $\bigcup_{i=1}^{N}(A+a_{i})\subseteq B\cap G$. On the other
hand, as $A\supseteq D\cap G$, we also have $\bigcup_{i=1}^{N}(A+a_{i})\supseteq B\cap G$.
\end{proof}

\begin{lem}\label{lem:From_union_of_cheeses_of_different_radiuses_to_a_single_radius}
Let $A=G\cap\bigsqcup_{i=1}^{n}F_{i}$ where $G$ is a dense subgroup
of $\mathcal{M}$ and $\{F_{i}\}_{i=1}^{n}$ are disjoint proper
cheeses with nonstandard radiuses. Then there are $N,m\in\N$ and
$c_{1},\dots,c_{N}\in G$ such that $\bigcap_{i=1}^{N}(A-c_{i})=G\cap\bigsqcup_{i=1}^{m}P_{i} $
with $P_{i} $ \uline{pseudo}  proper cheeses with disjoint
outer balls, all of the same nonstandard radius, and $0\in P_{1} $.
 \end{lem}
 
\begin{proof}
It is of course enough to prove the lemma without the requirement
$0\in P_{1} $, as we may then arrange that by shifting by some
$c\in G\cap P_{1} $.

\emph{Preparation step.} By \remref{splitting_swiss_cheeses} \textit{(2)}, if $F$ is a proper cheese of infinite radius $\gamma$ then, for all $k\geq 0$, $F$ can be written as a disjoint union of proper cheeses of radius $\gamma+k$. So there exists $\gamma_1,\dots,\gamma_n$, in distinct archimedean classes of $\Gamma$, such that
we can write
$$\bigsqcup_{i = 1} ^{n} F_i = \bigsqcup_{i=1}^m\bigsqcup_{j = 1}^{s_i} F_j^i,$$
where $s_1,\dots, s_m \geq 1$ and  for all $1\leq i\leq m$ and $1\leq j\leq s_i$, $rad(F_j^i) = \gamma_i$ and $F_j^i$ has a swiss cheese representation in which the radiuses of all the holes are in $$R:= \set{\alpha\in \Gamma : \text{for all $1\leq k\leq m$, if $\abs{\alpha - \gamma_k}\in \N$ then $\alpha\leq \gamma_k$}}.$$
We call this representation of A a \emph{good representation of A with respect to $\set{\gamma_i}_{i=1}^{m}$}.

~

If $m=1$, we already have what we
want, so we may assume that $m>1$. For each $i,j$, let $B_{j}^{i}$ be the outer ball of $F_{j}^{i}$. There
are two cases:

~

\emph{Case 1:} For every $1 < l\leq m$ and every $1\leq u\leq s_{l}$
there is some $1\leq v \leq s_{1}$ such that $B_{v}^{1}\cap B_{u}^{l}\neq\emptyset$.

This means that $\{B_{j}^{1}\}_{j=1}^{s_{1}}$
is the set of all the maximal balls with respect to inclusion among
$\{B_{j}^{i}\,:\, 1\leq i\leq m\,,\,1\leq j\leq s_{i}\}$. It follows that  $\{B_{j}^{1}\}_{j=1}^{s_{1}}$ are outer balls of pseudo proper cheese containing all the $F_j^i$. Indeed,
by the proof of Lemma~\ref{lem:about_unions_of_pseudo_cheeses} \textit{(2)},
we may write $$\bigsqcup_{i=1}^{m}\bigsqcup_{j=1}^{s_{i}}F_{j}^{i}=\bigsqcup_{j=1}^{s_{1}}P_{j},$$
where for each $j$, $P_{j}$ is a pseudo proper cheese such that $F_{j}^{1}\subseteq P_{j}\subseteq B_{j}^{1}$.
So these are pseudo proper cheeses with disjoint outer balls, all
of the same radius $\gamma_{1}$. So in this case we are done. 

~

\emph{Case 2:} There are $1< l\leq m$ and $1\leq v \leq s_{l}$
such that for every $1\leq j\leq s_{1}$, $B_{j}^{1}\cap B_{v}^{l}=\emptyset$.

Let $a\in F_{1}^{1}\cap G$ and $b\in F_{v}^{l}\cap G$ and set
$A^{\tg}=(A-a)\cap(A-b)$. Then $0\in A^{\tg}\neq\emptyset$. We show that $A'$ has a good representation with respect to a subset of $\set{\gamma_i}_{i=1}^{m}$, of the form $$A^{\tg}=G\cap\bigsqcup_{i=1}^{m'}\bigsqcup_{j=1}^{s_{i}^{\tg}}{\tilde{F}}_{j}^{i}$$
 such that either there are no more proper cheeses of radius $\gamma_1$, or the number $s_1'$ of proper cheeses of radius $\gamma_1$ is strictly less than $s_1$. By reiterating this process, it will terminate either to the case in which every proper cheese is of the same radius or to \emph{Case 1}, which proves the Lemma.

Write $A^{\tg}=G\cap(\bigsqcup_{i=1}^{m}\bigsqcup_{j=1}^{s_{i}}\bigsqcup_{q=1}^{m}\bigsqcup_{r=1}^{s_{i}}(F_{j}^{i}-a)\cap(F_{r}^{q}-b))$.
By the good representation, for each $i,j$ we write $F_{j}^{i}=B_{j}^{i}\backslash\bigsqcup_{t}B_{j,t}^{i}$
with $rad(B_{j,t}^{i})\in R$.

For every $i$ and $j,k$, if $B_{j}^{i}-a\neq B_{k}^{i}-b$,
then $(F_{j}^{i}-a)\cap(F_{k}^{i}-b)=\emptyset$, and if $B_{j}^{i}-a=B_{k}^{i}-b$,
then $(F_{j}^{i}-a)\cap(F_{k}^{i}-b)$ is a proper cheese
of radius $\gamma_{i}\geq\gamma_{1}$ such that all its holes can be written with radiuses in $R$. 

For every $i<i'$ and $j,k$, if $(B_{j}^{i}-a)\cap(B_{k}^{i^{\tg}}-b)=\emptyset$,
then also $(F_{j}^{i}-a)\cap(F_{k}^{i^{\tg}}-b)=\emptyset$.
Otherwise, $(B_{j}^{i}-a)\supseteq(B_{k}^{i^{\tg}}-b)$ and

$$(F_{j}^{i}-a)\cap(F_{k}^{i^{\tg}}-b)=((B_{k}^{i^{\tg}}-b)\backslash\bigsqcup_{t^{\tg}}(B_{k,t^{\tg}}^{i^{\tg}}-b))\backslash\bigsqcup_{t}(B_{j,t}^{i}-a).$$
For each $t$ such that $(B_{j,t}^{i}-a)\cap(B_{k}^{i^{\tg}}-b)\neq\emptyset$
there are three cases:
\begin{enumerate}
\item $rad(B_{k}^{i^{\tg}}-b)>rad(B_{j,t}^{i}-a)$. Then $(B_{k}^{i^{\tg}}-b)$ is included in the hole $(B_{j,t}^{i}-a)$ hence $(F_{j}^{i}-a)\cap(F_{k}^{i^{\tg}}-b)=\emptyset$. 
\item $rad(B_{k}^{i^{\tg}}-b)\leq rad(B_{j,t}^{i}-a)$ and $rad(B_{j,t}^{i}-a)$ is at finite distance from $\gamma_{i'}$.
As $rad(B_{j,t}^{i}-a)=rad(B_{j,t}^{i})\in  R$, we get $$rad(B_{k}^{i^{\tg}}-b)=rad(B_{k}^{i^{\tg}})=\gamma_{i^{\tg}}\geq rad(B_{j,t}^{i}-a).$$
So $rad(B_{k}^{i^{\tg}}-b)=rad(B_{j,t}^{i}-a)$, and so $(B_{k}^{i^{\tg}}-b)=(B_{j,t}^{i}-a)$
and therefore $(F_{j}^{i}-a)\cap(F_{k}^{i^{\tg}}-b)=\emptyset$. 
\item $rad(B_{k}^{i^{\tg}}-b)\leq rad(B_{j,t}^{i}-a)$ and $rad(B_{j,t}^{i}-a)$ is not at finite distance from $\gamma_{i'}$.
Then $B_{j,t}^{i}-a$ is a proper hole of $(F_{j}^{i}-a)\cap(F_{k}^{i^{\tg}}-b)$. 
\end{enumerate}
Therefore $(F_{j}^{i}-a)\cap(F_{k}^{i^{\tg}}-b)$ is either
empty or a proper cheese of radius $\gamma_{i^{\tg}}>\gamma_{i}\geq\gamma_{1}$
such that all its holes can be written with radiuses in $R$.

~

So $A'$ has a good representation that is the intersection of $G$ with a (nonempty) disjoint union
of proper cheeses, with radiuses among \emph{$\{\gamma_{i}\}_{i=1}^{m}$},
such that all their holes have radiuses in $R$. Now either $s_1 = 1$, hence $F_1^1$ is the only cheese of radius $\gamma_1$ in the good representation of $A$ and hence in the good representation of $A'$ there are no more proper cheeses of radius $\gamma_1$. Otherwise we have a good representation with respect to a subset of $\set{\gamma_i}_{i=1}^m$ of the form$$A^{\tg}=G\cap\bigsqcup_{i=1}^{m'}\bigsqcup_{j=1}^{s_{i}^{\tg}}\tilde{F}_{j}^{i}$$
where $s_1',\dots,s_{m'}' \geq 1$, and $s_1'$ is the number of cheese of radius $\gamma_1$. For
every $1\leq l \leq s_{1}^{\tg}$, there must be $j,k$ such
that $\tilde{F}_{l}^{1}=(F_{j}^{1}-a)\cap(F_{k}^{1}-b)$.
As $(F_{j}^{1}-a)\cap(F_{k}^{1}-b)\neq\emptyset\iff B_{j}^{1}-a=B_{k}^{1}-b$,
for every $j$ there is at most one $k$ such that $(F_{j}^{1}-a)\cap(F_{k}^{1}-b)\neq\emptyset$,
therefore $s_{1}^{\tg}\leq s_{1}$. Suppose towards contradiction
that $s_{1}^{\tg}=s_{1}$. Then for every $j$ there is exactly one
$k$ such that $(F_{j}^{1}-a)\cap(F_{k}^{1}-b)\neq\emptyset$,
in particular, for $j=1$ there is exactly one $l$ such that
$(F_{1}^{1}-a)\cap(F_{l}^{1}-b)\neq\emptyset$, and so also $B_{1}^{1}-a=B_{l}^{1}-b$.
By the choice of $a,b$, we have $0\in(B_{1}^{1}-a)\cap(B_{v}^{l}-b)=(B_{l}^{1}-b)\cap(B_{v}^{l}-b)$,
so $b\in B_{l}^{1}\cap B_{v}^{l}\neq\emptyset$, a contradiction.
Therefore $s_{1}^{\tg}<s_{1}$.
\end{proof}
~

\begin{lem}
\label{lem:From_union_of_cheeses_with_the_same_radius_to_single_cheese}Let
$A=G\cap\bigsqcup_{i=1}^{n}P_{i}$ where $G$ is a dense subgroup
of $\mathcal{M}$ and $\{P_{i}\}_{i=1}^{n}$ are pseudo proper cheeses
with disjoint outer balls, all of the same nonstandard radius $\alpha$,
such that $0\in P_{1}$. Then there exists $N\in\N$ and $c_{1},\dots,c_{N}\in G$
such that $\bigcap_{i=1}^{N}(A-c_{i})=G\cap P$ for some pseudo
proper cheese $P$ of nonstandard radius such that $0\in P$.\end{lem}
\begin{proof}
It is of course enough to prove the lemma without the requirement
$0\in P$. We proceed by induction on $n$. For $n=1$ we have nothing
to prove. Suppose that the lemma holds for all $n^{\tg}<n$. For each
$1\leq i\leq n$ let $B_{i}$ be the outer ball of $P_{i}$, and
let $F_{i}$ be a proper cheese with outer ball $B_{i}$ such that
$F_{i}\subseteq P_{i}\subseteq B_{i}$. Let $S$ be the set of all
the balls of radius $\alpha$, and let $S^{\tg}=\{B_{i}\,:\,1\leq i\leq n\}$.
Observe that $(S,+)$ is an infinite group with neutral element $B_{1}$
(since $0\in P_{1}\subseteq B_{1}$), and in particular, $S^{\tg}\subsetneq S$.
Let $C:=\bigcup S^{\tg}=\bigsqcup_{i=1}^{n}B_{i}$. 
\begin{claim*}
If for every $1\leq i\leq n$ there is $a\in B_{i}$ such that $S^{\tg}-a=S^{\tg}$,
then $S^{\tg}$ is a subgroup of $S$. \end{claim*}
\begin{proof}[Proof of claim]
If $B,B^{\tg}\in S$ then $rad(B)=rad(B^{\tg})$, hence $(B-a)\cap B^{\tg}\neq\emptyset\Rightarrow B-a=B^{\tg}$.
Also, for all $B^{\tg\tg}\in S$ and $a,a^{\tg}\in B^{\tg\tg}$, $a-a^{\tg}\in B_{1}$
and therefore $B-a^{\tg}=(B-a)+(a-a^{\tg})=B-a$. From this and the
hypothesis of the claim it follows that for each $1\leq i\leq n$,
$S^{\tg}-B_{i}:= \set{B-B_i : B\in S'} = S^{\tg}$, which implies that $S^{\tg}$ is a subgroup of $S$. 
\end{proof}
There are two cases:\\

\emph{Case 1}: $S^{\tg}$ is a subgroup of $S$. Then $(C,+)$ is
a subgroup of $(M,+)$, and $S^{\tg}$ is the quotient group $C/B_{1}$.
As $(C,+)$ is definable, by \lemref{Characterizing_definable_subgroups}
it must be of the form $C=B(0,\beta)$ (as $B_{1}\not\subseteq mM$
for every $m>1$ with $\mbox{gcd}(m,p)=1$). In fact, since $|S^{\tg}|=n$,
it must be that $\beta=\alpha-k$, where $k$ satisfies $n=p^{k}$.
In particular, $\beta$ is nonstandard. For each $i$, let $H_{i}$
be (any choice for) the set of holes of $F_{i}$, and let $H=\bigcup_{i}H_{i}$.
Then we can rewrite $\bigsqcup_{i=1}^{n}F_{i}$ as $F=B(0,\beta)\backslash\bigcup H$,
which is a single proper cheese, with outer ball $B(0,\beta)$.
Let $P=\bigsqcup_{i=1}^{n}P_{i}$. Then $F\subseteq P\subseteq B(0,\beta)$,
so $P$ is a pseudo proper cheese, and we are done.

~

\emph{Case 2}: $S^{\tg}$ is not a subgroup of $S$. Then by the claim,
there is some $1\leq i_{0}\leq n$ such that for all $a\in B_{i_{0}}$,
$S^{\tg}-a\neq S^{\tg}$ (in fact $1<i_{0}$). Let $a\in G\cap P_{i_{0}}\subseteq B_{i_{0}}$
(which exists because $G$ is dense), and let $A^{\tg}=A\cap(A-a)$.
Then $0\in A^{\tg}\neq\emptyset$. 

Write $A^{\tg}=G\cap(\bigsqcup_{i=1}^{n}\bigsqcup_{j=1}^{n}P_{i}\cap(P_{j}-a))$.
Then 
$$G\cap\bigsqcup_{i=1}^{n}\bigsqcup_{j=1}^{n}F_{i}\cap(F_{j}-a)\subseteq A^{\tg}\subseteq G\cap\bigsqcup_{i=1}^{n}\bigsqcup_{j=1}^{n}B_{i}\cap(B_{j}-a).$$
For all $1\leq i,j\leq n$, $rad(B_{i})=rad(B_{j})=\alpha$ and therefore,
as in \lemref{From_union_of_cheeses_of_different_radiuses_to_a_single_radius},
$B_{i}\cap(B_{j}-a)\neq\emptyset\iff B_{i}=B_{j}-a\iff F_{i}\cap(F_{j}-a)\neq\emptyset$,
and in this case, $F_{i}\cap(F_{j}-a)$ is a proper cheese with
outer ball $B_{i}$. We also have that $F_{i}\cap(F_{j}-a)\subseteq P_{i}\cap(P_{j}-a)\subseteq B_{i}\cap(B_{j}-a)$,
so $P_{i}\cap(P_{j}-a)\neq\emptyset\iff B_{i}\cap(B_{j}-a)\neq\emptyset$,
and in this case, $P_{i}\cap(P_{j}-a)$ is a pseudo proper cheese
with outer ball $B_{i}$. Therefore, $G\cap(\bigsqcup_{i=1}^{n}\bigsqcup_{j=1}^{n}B_{i}\cap(B_{j}-a))=G\cap(\bigsqcup_{i=1}^{n^{\tg}}B_{i}^{\tg})$,
$G\cap(\bigsqcup_{i=1}^{n}\bigsqcup_{j=1}^{n}F_{i}\cap(F_{j}-a))=G\cap(\bigsqcup_{i=1}^{n^{\tg}}F_{i}^{\tg})$,
and $A^{\tg}=G\cap(\bigsqcup_{i=1}^{n^{\tg}}P_{i}^{\tg})$, where
for each $i$, $B_{i}^{\tg}\in S^{\tg}$, $F_{i}^{\tg}$ is a proper
swiss cheese with outer ball $B_{i}^{\tg}$, and $P_{i}^{\tg}$ is
a pseudo proper cheese such that $F_{i}^{\tg}\subseteq P_{i}^{\tg}\subseteq B_{i}^{\tg}$.

Moreover, for every $i$ there is at most one $j$ such that $B_{i}\cap(B_{j}-a)\neq\emptyset$,
therefore $n^{\tg}\leq n$. But by the choice of $a$, $S^{\tg}-a\neq S^{\tg}$,
so there is an $1\leq i\leq n$ such that $B_{i}\neq B_{j}-a$ for
all $1\leq j\leq n$. Therefore $n^{\tg}<n$, and by the induction
hypothesis we are done.
\end{proof}
~ 

~

\subsection{Proof of the theorem}

To prove \thmref{our_result_unstable_p_adic_elementary_extension_version}
we first prove a lemma that enables us to reduce the problem to single
variable formulas. Recall the following:
\begin{fact}[{\cite[Theorem 2.13]{Shelah1990}}]
\label{fact:shelah_lemma}A theory $T$ is stable if and only if
all formulas $\phi(x,y)$ over $\emptyset$ with $|x|=1$ are stable.
\end{fact}
Using this, we can prove:
\begin{lem}
\label{lem:reduction_to_single_variable}Let $L$ be any language
and let $T$ be an unstable $L$-theory with monster model $\mathcal{M}$.
Let $L^{-}\subseteq L$ be such that $T|_{L^{-}}$ is stable. Then
there exists an $L$-formula $\phi(x,y)$ over $\emptyset$ with $|x|=1$
and $b\in\mathcal{M}$ such that $\phi(x,b)$ is not $L^{-}$-definable
with parameters in $\mathcal{M}$.\end{lem}
\begin{proof}
By \ref{fact:shelah_lemma} there is an unstable $L$-formula $\phi(x,y)$
over $\emptyset$ with $|x|=1$. Let $(a_{i})_{i\in\Z}$, $(b_{i})_{i\in\Z}$
be two indiscernible sequences in $\mathcal{M}$ witnessing the instability
of $\phi(x,y)$, i.e., $\phi(a_{i},b_{j})$ if and only if $i<j$. Assume towards
contradiction that $\phi(x,b_{0})$ is definable by an $L^{-}$-formula
$\psi(x,c_{0})$ with parameters $c_{0}$ in $\mathcal{M}$. For each
$k\in\Z\backslash\{0\}$, as $tp(b_{k}/\emptyset)=tp(b_{0}/\emptyset)$
there is an automorphism of $L$-structures $\sigma_{k}\in Aut(\mathcal{M}/\emptyset)$
such that $\sigma_{k}(b_{0})=b_{k}$. Let $c_{k}=\sigma_{k}(c_{0})$.
Then $\phi(x,b_{k})$ is equivalent to $\psi(x,c_{k})$, and hence
$\psi(a_{i},c_{j})$ if and only if $i<j$, a contradiction to the stability of
$T|_{L^{-}}$.
\end{proof}
~

Lemma \ref{lem:reduction_to_single_variable} allows us to
give a simple proof for the unstable case of \corref{conant_result_order_elementary_extension_version}:
\begin{thm}[Conant, Unstable case of \corref{conant_result_order_elementary_extension_version}]
\label{thm:conant_result_order_unstable_elementary_extension_version}Let
$(N,+,0,1,<)$ be an elementary extension of $(\Z,+,0,1,<)$. Then
$(N,+,0,1,<)$ is \textbf{$\emptyset$}-minimal among the \uline{unstable}
\textbf{$\emptyset$}-proper \textbf{$\emptyset$}-expansions of $(N,+,0,1)$.\end{thm}
\begin{proof}
Let $\cN$ be any unstable structure with universe $N$, which is
a \textbf{$\emptyset$}-proper \textbf{$\emptyset$}-expansion of
$(N,+,0,1)$ and a \textbf{$\emptyset$}-reduct of $(N,+,0,1,<)$.
We show that $\cN$ is \textbf{$\emptyset$}-interdefinable with $(N,+,0,1,<)$.
It is enough to show that $x\geq0$ is definable over $\emptyset$
in $\cN$. Let $L$ be the language of $\cN$, $L^{-}=\{+,0,1\}$
and $L_{<}=\{+,0,1,<\}$. We may expand all these languages by adding
the symbols $\{-\}\cup\{D_{n}\,:\, n\geq1\}$, as all of them are
already definable over $\emptyset$ in all three languages. As $\cN$
is a \textbf{$\emptyset$}-expansion of $(N,+,0,1)$ and a \textbf{$\emptyset$}-reduct
of $(N,+,0,1,<)$, we may replace $L$ with $L\cup L^{-}$ and $L_{<}$
with $L_{<}\cup L\cup L^{-}$ without adding new $\emptyset$-definable
sets to any structure. So we may assume that $L^{-}\subseteq L\subseteq L_{<}$. 

Let $\mathcal{M}$ be a monster model for $Th(\Z,+,0,1,<)$, so $\mathcal{M}|_{L}$
is a monster for $Th(\cN)$. As $(N,+,0,1)$ is stable but $\cN$
is not, by \lemref{reduction_to_single_variable} there exist an $L$-formula
$\phi(x,y)$ over $\emptyset$ with $|x|=1$ and $b\in\mathcal{M}$
such that $\phi(x,b)$ is not $L^{-}$-definable with parameters in
$\mathcal{M}$. By quantifier elimination in $Th(\Z,+,0,1,<)$ and \lemref{reduction_to_a_single_D_m_with_m_coprime_to_all_p_k}
(\ref{enu:reduction to a single D_m}) (which is a theorem of $Th(\Z,+,0,1)$),
$\phi(x,b)$ is equivalent to a formula of the form 
\[
\bigvee_{i}(D_{m_{i}}(x-k_{i})\wedge x\in[c_{i},c_{i}^{\tg}])
\]
where $c_{i},c_{i}^{\tg}\in M\cup\{-\infty,+\infty\}$ and $[c_{i},c_{i}^{\tg}]$
denotes the closed interval except if one of the bound is infinite, in which case it is open on the infinite side. Let $m=\prod_{i}m_{i}$.
As each formula of the form $D_{m_{i}}(x-k)$ is equivalent to a disjunction
of formulas of the form $D_{m}(x-k^{\tg})$, we can rewrite this as
\[
\bigvee_{i}(D_{m}(x-k_{i})\wedge x\in[c_{i},c_{i}^{\tg}])
\]
(with possibly different $k_{i}$'s and numbering). By grouping together
disjuncts with the same $k_{i}$, we can rewrite this as 
\[
\bigvee_{i}(D_{m}(x-k_{i})\wedge\bigvee_{j}x\in[c_{i,j},c_{i,j}^{\tg}])
\]
where for $i_{1}\neq i_{2}$, $k_{i_{1}}\not\equiv k_{i_{2}}\mbox{ mod }m$.
As this formula is equivalent to $\phi(x,b)$, which is not $L^{-}$-definable
with parameters in $\mathcal{M}$, there must be an $i_{0}$ such
that $D_{m}(x-k_{i_{0}})\wedge\bigvee_{j}x\in[c_{i_{0},j},c_{i_{0},j}^{\tg}]$
is not $L^{-}$-definable with parameters in $\mathcal{M}$. This
latter formula, which we denote by $\phi_{i_{0}}(x,b)$, is equivalent
to $\phi(x,b)\wedge D_{m}(x-k_{i_{0}})$, and so is $L$-definable.
Let $\psi(x,b)$ be the formula $\phi_{i_{0}}(mx+k_{i_{0}},b)$. Then
$\psi(x,b)$ is $L$-definable and equivalent to just $\bigvee_{j}mx+k_{i_{0}}\in[c_{i_{0},j},c_{i_{0},j}^{\tg}]$.
This substitution is reversible as $\phi_{i_{0}}(x,b)$ is equivalent
to $D_{m}(x-k_{i_{0}})\wedge\psi(\frac{x-k_{i_{0}}}{m},b)$, therefore
also $\psi(x,b)$ is not $L^{-}$-definable with parameters in $\mathcal{M}$.
Each formula of the form $mx+k\in[c,c^{\tg}]$ is equivalent to the
formula $x\in[\lceil\frac{c-k}{m}\rceil,\lfloor\frac{c^{\tg}-k}{m}\rfloor]$,
so we can rewrite $\psi(x,b)$ as $\bigvee_{i=1}^{n}x\in[c_{i},c_{i}^{\tg}]$.
By reordering and combining intersecting intervals, we may assume
that the intervals are disjoint and increasing, i.e., for all $i<n$,
$c_{i}^{\tg}< c_{i+1}$. 

~

Now we show how from $\psi(x,b)$ we can get an $L$-definable formula
equivalent to $[0,a]$, for $a$ a positive nonstandard integer in
$\mathcal{M}$. For each $i$, if $[c_{i},c_{i}^{\tg}]$ defines in
$\mathcal{M}$ a finite set then it is $L^{-}$-definable, and so
$\psi(x,b)\wedge\neg x\in[c_{i},c_{i}^{\tg}]$ is also $L$-definable
but not $L^{-}$-definable (since $(\psi(x,b)\wedge\neg x\in[c_{i},c_{i}^{\tg}])\vee x\in[c_{i},c_{i}^{\tg}]$
is again equivalent to $\psi(x,b)$). So we may assume that for all
$i$, $[c_{i},c_{i}^{\tg}]$ is infinite. Note that as $\psi(x,b)$
is not $L^{-}$-definable, it cannot be empty. 

We want $\psi(x,b)$ to have a lower bound, i.e., $-\infty<c_{1}$.
If $c_{1}=-\infty$ but $c_{n}^{\tg}\neq+\infty$, then we can just
replace $\psi(x,b)$ with $\psi(-x,b)$. If both $c_{1}=-\infty$
and $c_{n}^{\tg}=+\infty$, we can replace $\psi(x,b)$ with $\neg\psi(x,b)$
and again remove all finite intervals. In both cases, $\psi(x,b)$
is still $L$-definable but not $L^{-}$-definable, so it is still
a nonempty disjunction of infinite disjoint intervals. 

By replacing $\psi(x,b)$ with $\psi(x+c_{1},b)$ we may assume
that $c_{1}=0$, so the leftmost interval is $[0,c_{1}^{\tg}]$. If $c_{1}^{\tg}\neq+\infty$
let $a^{\tg}=c_{1}^{\tg}$, otherwise let $a^{\tg}\in\mathcal{M}$
be any positive nonstandard integer. Let $\theta(x,b^{\tg})$ denote
the formula $\psi(x,b)\wedge\psi(a^{\tg}-x,b)$. Then $\theta(x,b^{\tg})$
is $L$-definable and equivalent to the infinite interval $[0,a^{\tg}]$. The proof of the following claim is an obvious consequence of quantifier elimination for Presburger arithmetic and is left to the reader.
\begin{claim}
\label{claim:first_claim_end_proof_order}For every $c\geq0$ there
exist $a>c$ and $b$ such that $\theta(x,b)$ is equivalent to the
interval $[0,a]$.\end{claim}
In particular, as $N$ is a small subset of $\cM$, there exists $c\in\cM$
bigger than all elements of $N$. By the claim, there exist $\tilde{a}>c$
and $\tilde{b}$ such that $\theta(x,\tilde{b})$ is equivalent to
the interval $[0,\tilde{a}]$, and so $\theta(N,\tilde{b})=\{s\in N\,:\, s\geq0\}$.

~

Let $\chi(y,z)$ be the formula $\chi_{1}(y,z)\wedge\chi_{2}(y,z)\wedge\chi_{3}(y,z)$
where:
\begin{itemize}
\item $\chi_{1}(y,z)$ is the formula $\theta(0,z)\wedge\theta(y,z)\wedge\neg\theta(-1,z)\wedge\neg\theta(y+1,z)\wedge\neg\theta(2y,z)$. 
\item $\chi_{2}(y,z)$ is the formula $\forall w((w\neq0\wedge\theta(w,z))\rightarrow\theta(w-1,z))$.
\item $\chi_{3}(y,z)$ is the formula $\forall w((w\neq y\wedge\theta(w,z))\rightarrow\theta(w+1,z))$.
\end{itemize}
So $\chi(y,z)$ is $L$-definable over $\emptyset$.
\begin{claim}
\label{claim:second_claim_end_proof_order}For every $a,b\in\cM$,
$\mathcal{M}\vDash\chi(a,b)$ if and only if $a>0$ and $\theta(\cM,b)=[0,a]$.\end{claim}
\begin{proof}
This can be formulated as a first order sentence in $L_{<}$ without
parameters:
\[
\cM\vDash\forall y,z(\chi(y,z)\leftrightarrow(y>0\wedge\forall x(\theta(x,z)\leftrightarrow0\leq x\leq y))),
\]
so it is enough to prove this for $\Z$. Let $a,b\in\Z$. If $a>0$
and $\theta(\Z,b)=[0,a]$, then clearly $\mathcal{\Z}\vDash\chi(a,b)$.
Suppose $\mathcal{\Z}\vDash\chi(a,b)$, and denote $A:=\theta(\Z,b)$.
By $\chi_{1}$, $0,a\in A$ and $-1,a+1,2a\notin A$. Suppose towards
contradiction that $a<0$. Then from $\chi_{2}$ it follows by induction
that $(-\infty,a]\subseteq A$. But then $2a\in A$, a contradiction.
So $a\geq0$. If $a=0$ then again $2a\in A$ is a contradiction.
So $a>0$. From $\chi_{2}$ it follows by induction that $[0,a]\subseteq A$.
Also, from $a+1\notin A$ and $\chi_{2}$ it follows by induction
that $[a+1,\infty)\cap A=\emptyset$, and from $-1\notin A$ and $\chi_{3}$
it follows by induction that $(-\infty,-1]\cap A=\emptyset$. So $A=[0,a]$.
\end{proof}
Now, let $\delta(x)$ be the formula 
\[
\exists y,z(\chi(y,z)\wedge\theta(x,z))\mbox{.}
\]
Then $\delta(x)$ is $L$-definable over $\emptyset$, and we claim
that it defines $x\geq0$ in $\cN$: For $s\in N$, if $\cN\vDash\delta(s)$
then there are $a,b\in N$ such that $\cN\vDash\chi(a,b)\wedge\theta(s,b)$,
so by \claimref{second_claim_end_proof_order}, $s\in[0,a]$ hence
$s\geq0$. On the other hand, suppose $s\geq0$. By the choice of
$\tilde{a},\tilde{b}$, $\cM\vDash\chi(\tilde{a},\tilde{b})\wedge\theta(s,\tilde{b})$,
so $\cM\vDash\delta(s)$, and by elementarity, $\cN\vDash\delta(s)$.
Therefore, $x\geq0$ is definable over $\emptyset$ in $\cN$.
\end{proof}
~
\begin{rem}
\label{rem:any_order_intermediate_with_new_one_dimensional_set_is_unstable}The
part in the proof where we start with an $L$-formula $\phi(x,y)$
over $\emptyset$ with $|x|=1$ and $b\in\mathcal{M}$ such that $\phi(x,b)$
is not $L^{-}$-definable with parameters in $\mathcal{M}$, and show
that there exists a formula $\theta(x,b^{\tg})$ which is $L$-definable
and equivalent to the infinite interval $[0,a^{\tg}]$, works the
same for any structure $\cN$ which is a proper expansion of $(N,+,0,1)$
and a reduct of $(N,+,0,1,<)$. $\cN$ does not have to be a \textbf{$\emptyset$}-expansion
of $(N,+,0,1)$ or a \textbf{$\emptyset$}-reduct of $(N,+,0,1,<)$,
nor unstable, as long as such $\phi(x,y)$ and $b$ exist (being a
\textbf{$\emptyset$}-reduct is needed in the proof for $\phi(x,y)$
to also be $\emptyset$-definable in $L_{<}$). So in any structure
$\cN$ which is a proper expansion of $(N,+,0,1)$ and a reduct of
$(N,+,0,1,<)$, and which has a definable one-dimensional set which
is not definable in $(N,+,0,1)$, there exists a definable infinite
interval, and hence it is unstable.
\end{rem}
~

Combined with \factref{conant_pillay_stable_expansions_finite_dp_rank},
we recover \corref{conant_result_order_elementary_extension_version}
and \thmref{conant_result_order}:
\begin{proof}[Proof of \corref{conant_result_order_elementary_extension_version}]
Suppose for a contradiction that there exists a structure $\cN$
with universe $N$, which is a \textbf{$\emptyset$}-proper \textbf{$\emptyset$}-expansion
of $(N,+,0,1)$ and a \textbf{$\emptyset$}-proper \textbf{$\emptyset$}-reduct
of $(N,+,0,1,<)$. So $\cN$ is dp-minimal, and by \thmref{conant_result_order_unstable_elementary_extension_version},
it must also be stable. By \obsref{reducts_preserved_by_elementary_substructures},
relativization to $\Z$ gives us a structure $\cZ\prec\cN$ which
is a \textbf{$\emptyset$}-proper \textbf{$\emptyset$}-expansion
of $(\Z,+,0,1)$ and a \textbf{$\emptyset$}-proper \textbf{$\emptyset$}-reduct
of $(\Z,+,0,1,<)$. As every element of $(\Z,+,0,1)$ is $\emptyset$-definable,
a reduct of $(\Z,+,0,1)$ is in fact a \textbf{$\emptyset$}-reduct,
and so a \textbf{$\emptyset$}-proper \textbf{$\emptyset$}-expansion
of $(\Z,+,0,1)$ is in fact a proper \textbf{$\emptyset$}-expansion
of $(\Z,+,0,1)$, which is of course a proper expansion. So $\cZ$
is a stable dp-minimal proper expansion of $(\Z,+,0,1)$, a contradiction
to \factref{conant_pillay_stable_expansions_finite_dp_rank}.
\end{proof}
~
\begin{proof}[Proof of Theorem~\ref{conant_result_order}]
Suppose for a contradiction that there exists a structure $\cZ$
with universe $\Z$, which is a proper expansion of $(\Z,+,0,1)$
and a proper reduct of $(\Z,+,0,1,<)$. In $\cZ$, $+$, $0$, and
$1$ are definable, but not necessarily $\emptyset$-definable. We
expand $\cZ$ to a structure $\cZ^{\tg}$ by adding $+$, $0$, and
$1$ to the language. So $\cZ^{\tg}$ is a proper \textbf{$\emptyset$}-expansion
of $(\Z,+,0,1)$, and still a proper reduct of $(\Z,+,0,1,<)$. As
every element of $(\Z,+,0,1,<)$ is $\emptyset$-definable, a reduct
of $(\Z,+,0,1,<)$ is in fact a \textbf{$\emptyset$}-reduct. So $\cZ^{\tg}$
is a proper \textbf{$\emptyset$}-expansion of $(\Z,+,0,1)$, and
a proper \textbf{$\emptyset$}-reduct of $(\Z,+,0,1,<)$. As a proper
\textbf{$\emptyset$}-expansion/reduct is obviously a \textbf{$\emptyset$}-proper
\textbf{$\emptyset$}-expansion/reduct, this contradicts \corref{conant_result_order_elementary_extension_version}. 
\end{proof}
~

The proof of \thmref{our_result_unstable_p_adic_elementary_extension_version}
is similar, but more involved and relies on \subref{Cheese-Feng-Shui}.
\begin{proof}[Proof of \thmref{our_result_unstable_p_adic_elementary_extension_version}]
Let $\cN$ be any unstable structure with universe $N$, which is
a \textbf{$\emptyset$}-proper \textbf{$\emptyset$}-expansion of
$(N,+,0,1)$ and a \textbf{$\emptyset$}-reduct of $(N,+,0,1,|_{p})$.
We show that $\cN$ is \textbf{$\emptyset$}-interdefinable with $(N,+,0,1,|_{p})$.
It is enough to show that $x|_{p}y$ is definable over $\emptyset$
in $\cN$. Let $L$ be the language of $\cN$ and $L^{-}=\{+,0,1\}$.
As in the proof of \thmref{conant_result_order_unstable_elementary_extension_version},
we may assume that all languages contain $\{-\}\cup\{D_{n}\,:\, n\geq1\}$,
and (by being a \textbf{$\emptyset$}-reduct and \textbf{$\emptyset$}-expansion)
that $L^{-}\subseteq L\subseteq L_{p}^{E}$.

Let $\mathcal{M}$ be a monster model for $T_{p}$, so $\mathcal{M}|_{L}$
is a monster for $Th(\cN)$. As $(N,+,0,1)$ is stable but $\cN$
is not, by \lemref{reduction_to_single_variable} there exist an $L$-formula
$\phi(x,y)$ over $\emptyset$ with $|x|=1$ and $b\in\mathcal{M}$
such that $\phi(x,b)$ is not $L^{-}$-definable with parameters in
$\mathcal{M}$. By \thmref{QE} (quantifier elimination) and \remref{QE_one_variable_simpler_form},
$\phi(x,b)$ is equivalent to a formula of the form
\[
\bigvee_{i}\left(D_{m}(x-r_{i})\wedge kx\in F_{i} \wedge \bigwedge_{j}k^{\tg}x\neq a_{i,j} \right)\vee\bigvee_{i^{\tg}}x=c_{i^{\tg}}
\]
where $m,k,k^{\tg},r_{i}\in\Z$, $\mbox{gcd}(m,p)=\mbox{gcd}(k,p)=1,$
$k^{\tg}=p^{l}  k$ for some $l\geq0$, $a_{i,j},c_{i^{\tg}}\in\mathcal{M}$
and each $F_{i}$ is a swiss cheese in $\mathcal{M}$. 

The first step of the proof is to show the existence of an $L$-definable
formula which is equivalent to a formula of the form $D_{m}(x)\wedge x\in B(0,\gamma)$,
i.e. $D_{m}(x)\wedge v(x)\geq\gamma$, for some nonstandard $\gamma\in \Gamma$
and integer $m$ such that $\mbox{gcd}(m,p)=1$. Let $\phi^{\tg}(x,b)$ be the formula 
\[
\bigvee_{i}(D_{m}(x-r_{i})\wedge kx\in F_{i})\mbox{.}
\]
The symmetric difference $\phi(x,b)\triangle\phi^{\tg}(x,b)$ is finite,
hence $L^{-}$-definable, and therefore $\phi^{\tg}(x,b)$ is also
$L$-definable but not $L^{-}$-definable. So we may replace $\phi(x,b)$
by $\phi^{\tg}(x,b)$. For each $i$, the formula $D_{m}(x-r_{i})$
is equivalent to $D_{k  m}(kx-k  r_{i})$, so $\phi(x,b)$
is equivalent to the formula 
\[
\bigvee_{i}(D_{k  m}(kx-k  r_{i})\wedge kx\in F_{i})\mbox{.}
\]
Let $\phi^{\tg}(x,b)$ be the formula $D_{k}(x)\wedge\phi(\frac{x}{k},b)$.
Then $\phi^{\tg}(x,b)$ is $L$-definable and equivalent to the formula
\[
\bigvee_{i}(D_{m^{\tg}}(x-r_{i}^{\tg})\wedge x\in F_{i})
\]
where $m^{\tg}=k  m$ and $r_{i}^{\tg}=k  r_{i}$. This substitution
is reversible as $\phi(x,b)$ is equivalent to $\phi^{\tg}(kx,b)$,
therefore also $\phi^{\tg}(x,b)$ is not $L^{-}$-definable. So again
we may replace $\phi(x,b)$ by $\phi^{\tg}(x,b)$.

We want each $F_{i}$ to have a nonstandard radiuses. For each $i$,
choose a representation for $F_{i}$ as a swiss cheese $F_{i}=B_{i,0}\backslash\bigcup_{j=1}^{n_{i}}B_{i,j}$,
where $B_{i,j}=B(a_{i,j},\gamma_{i,j})$. Let $J_{i}=\{1\leq j\leq n_{i}\,:\,\gamma_{i,j}\notin\N\}$,
i.e., the set of indices of the infinite holes, and let $$B_{i,0}^{\tg}=\begin{cases}
B(0,0) & \gamma_{i,0}\in\N\\
B_{i,0} & \gamma_{i,0}\notin\N
\end{cases}\text{ and } B_{i,0}^{\tg\tg}=\begin{cases}
B_{i,0} & \gamma_{i,0}\in\N\\
B(0,0) & \gamma_{i,0}\notin\N
\end{cases}$$(note that $B(0,0)=M$). Let $F_{i}^{\tg}=B_{i,0}^{\tg}\backslash\bigcup_{j\in J_{i}}B_{i,j}$,
and let $F_{i}^{\tg\tg}=B_{i,0}^{\tg\tg}\backslash\bigcup_{j\notin J_{i}}B_{i,j}$.
Then $F_{i}=F_{i}^{\tg}\cap F_{i}^{\tg\tg}$, and so $\phi(x,b)$
is equivalent to 
\[
\bigvee_{i}(D_{m^{\tg}}(x-r_{i}^{\tg})\wedge x\in F_{i}^{\tg\tg} \wedge x\in F_{i}^{\tg}) \mbox{.}
\]
 Each hole of $F_{i}^{\tg}$ has nonstandard radius, and its outer
ball either has an nonstandard radius or has radius $0$. On the other
hand, both the outer ball and all the holes of $F_{i}^{\tg\tg}$
have finite radiuses. In general, if $B(a,\gamma)$ has finite radius,
then the formula $x\in B(a,\gamma)$ is equivalent to $D_{p^{\gamma}}(x-a)$.
So $x\in F_{i}^{\tg\tg}$ is equivalent to a boolean combination of
such formulas, and therefore, by \lemref{reduction_to_a_single_D_m_with_m_coprime_to_all_p_k}
(\ref{enu:reduction to a single D_m}) (choosing the same $m^{\tg\tg}$
for all the $i$'s and rearranging the disjunction), $\phi(x,b)$
is equivalent to a formula of the form 
\[
\bigvee_{i}(D_{m^{\tg\tg}}(x-r_{i}^{\tg})\wedge x\in F_{i}^{\tg})
\]
where each hole of $F_{i}^{\tg}$ has a nonstandard radius, and
its outer ball either has an nonstandard radius or has radius $0$.
Note that now it may be that $p|m^{\tg\tg}$. By grouping together
disjuncts with the same $r_{i}^{\tg}$, we can rewrite this as 
\[
\bigvee_{i}(D_{m^{\tg\tg}}(x-r_{i}^{\tg})\wedge\bigvee_{j}x\in F_{i,j}^{\tg})
\]
where for $i_{1}\neq i_{2}$, $r_{i_{1}}^{\tg}\not\equiv r_{i_{2}}^{\tg}\mbox{ mod }m^{\tg\tg}$.
As this formula is equivalent to $\phi(x,b)$, which is not $L^{-}$-definable
with parameters in $\mathcal{M}$, there must be an $i_{0}$ such
that $D_{m^{\tg\tg}}(x-r_{i_{0}}^{\tg})\wedge\bigvee_{j}x\in F_{i_{0},j}^{\tg}$
is not $L^{-}$-definable with parameters in $\mathcal{M}$. This
latter formula, which we denote by $\phi_{i_{0}}(x,b)$, is equivalent
to $\phi(x,b)\wedge D_{m^{\tg\tg}}(x-r_{i_{0}}^{\tg})$, and so is
$L$-definable. So we may replace $\phi(x,b)$ by $\phi_{i_{0}}(x,b)$.
For simplicity of notation we rewrite this as
\[
D_{m}(x-r)\wedge\bigvee_{i}x\in F_{i}\mbox{.}
\]
By \lemref{union_of_cheeses_with_nonempty_intersection_is_also_a_cheese}
we may assume that $\{F_{i}\}_{i}$ are pairwise disjoint, and still
have that for each $i$, all the holes of $F_{i}$ have nonstandard radiuses
and its outer ball either has a nonstandard radius or has radius $0$.
By \remref{A-finite-number-of-proper-holes-cannot-cover-a-ball} two
proper cheeses having the same outer ball must intersect. Applying
this to all the $F_{i}$'s having radius $0$ (which are all proper,
as all the holes are of nonstandard radius), we see that there can be
at most one $i$ such that $F_{i}$ has radius $0$. 

We want all proper cheeses to have nonstandard radius. If there is $i_{0}$
such that the proper cheese $F_{i_{0}}$ has radius $0$, let $\phi^{\tg}(x,b)$
be the formula $D_{m}(x-r)\wedge\neg\phi(x,b)$. Then $\phi^{\tg}(x,b)$
is $L$-definable and, as $\phi(x,b)$ is equivalent to $D_{m}(x-r)\wedge\neg\phi^{\tg}(x,b)$,
it is also not $L^{-}$-definable. The formula $\phi^{\tg}(x,b)$
is equivalent to 
\[
D_{m}(x-r)\wedge\bigwedge_{i}x\in F_{i}^{c}\mbox{.}
\]

We may write $F_{i_{0}}=B(0,0)\backslash\bigcup_{j=1}^{n}B_{j}$,
where for each $j$, $rad(B_{j})$ is nonstandard. So $F_{i_{0}}^{c}=\bigcup_{j=1}^{n}B_{j}$,
and $\phi^{\tg}(x,b)$ is equivalent to 

\[
D_{m}(x-r)\wedge\bigvee_{j=1}^{n}(x\in B_{j}\wedge\bigwedge_{i\neq i_{0}}x\in F_{i}^{c})\mbox{.}
\]

For each $i\neq i_{0}$, $F_{i}^{c}$ is a finite union of swiss cheeses
(specifically, a union of a single swiss cheese of radius $0$ and
a finite number of balls). Therefore, by \remref{splitting_swiss_cheeses} \textit{(4)},
for each $j$, $B_{j}\cap\bigcap_{i\neq i_{0}}F_{i}^{c}$ is a finite
union of swiss cheeses, each of radius at least $rad(B_{j})$, so
nonstandard. So $\phi^{\tg}(x,b)$ is equivalent to a formula of the
form 

\[
D_{m}(x-r)\wedge\bigvee_{i}x\in F_{i}^{\tg}
\]
where each $F_{i}^{\tg}$ is a swiss cheese of nonstandard radius. Again by
\lemref{union_of_cheeses_with_nonempty_intersection_is_also_a_cheese},
we may assume in addition that $\{F_{i}^{\tg}\}_{i}$ are pairwise
disjoint. As $\phi^{\tg}(x,b)$ is not $L^{-}$-definable, the disjunction
cannot be empty. So we may replace $\phi(x,b)$ by $\phi^{\tg}(x,b)$
and rename $F_{i}^{\tg}$ as $F_{i}$. 

We may assume that for each $i$, $D_{m}(x-r)\wedge x\in F_{i}$ defines
a nonempty set, as otherwise we may just drop the $i$'th disjunct.
Write $m=p^{k}  m^{\tg}$ with $\mbox{gcd}(m^{\tg},p)=1$. Then
$D_{m}(x-r)$ is equivalent to $D_{m^{\tg}}(x-r_{1})\wedge(v_{p}(x-r_{2})\geq k)$,
where $r_{1}=r\mbox{ mod }m'$ and $r_{2}=r\mbox{ mod }p^{k}$. So
$\phi(x,b)$ is equivalent to 
\[
D_{m^{\tg}}(x-r_{1})\wedge\bigvee_{i}(v_{p}(x-r_{2}\geq k)\wedge x\in F_{i})\mbox{.}
\]
The formula $v_{p}(x-r_{2})\geq k$ defines the ball $B(r_{2},k)$,
of finite radius $k$, and for each $i$, the outer ball of $F_{i}$
has a nonstandard radius. As $D_{m}(x-r)\wedge x\in F_{i}$ defines
a nonempty set, so too does $v_{p}(x-r_{2})\geq k\wedge x\in F_{i}$,
and hence the outer ball of $F_{i}$ is contained in $B(r_{2},k)$.
Therefore $v_{p}(x-r_{2})\geq k\wedge x\in F_{i}$ is equivalent
to just $x\in F_{i}$, and so $\phi(x,b)$ is equivalent to 
\[
D_{m^{\tg}}(x-r_{1})\wedge\bigvee_{i}x\in F_{i}\mbox{.}
\]
By \remref{splitting_swiss_cheeses} (\ref{enu:splitting_to_proper})
we may assume that each $F_{i}$ is a proper cheese. We replace $\phi(x,b)$
by $\phi(x+r_{1},b)$, and rename $m^{\tg}$ as $m$ and each $F_{i}-r_{1}$
as $F_{i}$. Altogether, $\phi(x,b)$ is equivalent to a formula of
the form 
\[
D_{m}(x)\wedge\bigvee_{i}x\in F_{i}
\]
where $\mbox{gcd}(m,p)=1$, and $\{F_{i}\}_{i}$ are disjoint proper cheeses having nonstandard radiuses. As $\phi(x,b)$ is not $L^{-}$-definable,
the disjunction cannot be empty. 

By \remref{D_m_with_m_coprime_to_p_are_dense}, $D_{m}(x)$ defines
a dense subgroup of $\mathcal{M}$. By successively applying Lemmas
\ref{lem:From_union_of_cheeses_of_different_radiuses_to_a_single_radius},
\ref{lem:From_union_of_cheeses_with_the_same_radius_to_single_cheese}
and \ref{lem:reduction_single_pseudo_cheese_to_ball}, we get an $L$-definable
formula of the form 
\begin{equation}
D_{m}(x)\wedge x\in B(0,\gamma)\tag{\ensuremath{\star}}\label{eq:phi_form_final}
\end{equation}
with $\gamma$ nonstandard and $\mbox{gcd}(m,p)=1$. We will now assume that $\phi(x,b)$ is of this form.\\
~

To finish, we need the following: 
\begin{claim}
\label{claim:finishing_claim_valuation}Let $\psi(x,z)$ be any $L_{p}$-formula
with $|x|=1$.
\begin{enumerate}
\item \label{enu:finishing_claim_valuation_part_1}Suppose there exists
$a\in\mathcal{M}$ with $v(a)$ nonstandard, for which there exists
$b$ such that $\psi(x,b)$ is equivalent to $v(x)\geq v(a)$. Then
for any $c$ such that $v(c)$ is nonstandard there is $b^{\tg}\in\mathcal{M}$
such that $tp(b^{\tg}/\emptyset)=tp(b/\emptyset)$ (in $L_{p}$) and
$\psi(x,b^{\tg})$ is equivalent to $v(x)\geq v(c)$. 
\item Let $\theta(z)$ be an $L_{p}$-formula. Then there exists $K\in\N$
such that for any $a\in\mathcal{M}$ with $v(a)\geq K$, if there
exists $b$ such that $\theta(b)$ holds and $\psi(x,b)$ is equivalent
to $v(x)\geq v(a)$, then for any $c$ such that $v(c)\geq K$ there
is $b^{\tg}\in\mathcal{M}$ such that $\theta(b^{\tg})$ and $\psi(x,b^{\tg})$
is equivalent to $v(x)\geq v(c)$. That is, let $\alpha(w)$ be the
formula defined by 
\[
\exists z(\theta(z)\wedge\forall x(\psi(x,z)\leftrightarrow v(x)\geq v(w)))
\]
 and let $\chi(w)$ be the formula defined by 
\[
\alpha(w)\rightarrow\forall w^{\tg}(v(w^{\tg})\geq K\rightarrow\alpha(w^{\tg})).
\]
Then $\chi(w)$ is satisfied by any $a$ such that $v(a)\geq K$.
\end{enumerate}
\end{claim}
\begin{proof}[Proof of claim]
~
Proof of (1). We show that we can find $a^{\tg}\in\mathcal{M}$ such that $tp(a^{\tg}/\emptyset)=tp(a/\emptyset)$
and $v(a^{\tg})=v(c)$. Indeed, let $\Sigma(x)$ be the partial type
$tp(a/\emptyset)\cup\{v(x)=v(c)\}$. We show that it is consistent.
Let $F\subseteq\Sigma(x)$ be a finite subset. As $v(a)$ is nonstandard,
we may assume that $F$ is of the form $$\{x\neq j\,:\,-n\leq j\leq n\}\cup\{D_{m_{k}}(x-r_{k})\,:\,1\leq k\leq s\}\cup\{v(x)=v(c)\}.$$
Let $m=\prod_{k}m_{k}$, and write $m=p^{l}  m^{\tg}$ with $\mbox{gcd}(m^{\tg},p)=1$.
By Lemma~\ref{lem:Preservation of a solution} \textit{(4)},
there exists $\tilde{a}\in\mathcal{M}$ satisfying the formula $D_{m^{\tg}}(x-a)\wedge(v(x)=v(c))$.
So $v(\tilde{a})=v(c)$ is nonstandard. As $v(a)$ is also nonstandard,
$\tilde{a}$ also satisfies $D_{p^{l}}(x-a)$, so it satisfies $D_{m}(x-a)$,
and therefore it satisfies $\{D_{m_{k}}(x-r_{k})\,:\,1\leq k\leq s\}$.
Also, as $v(\tilde{a})$ is nonstandard, $\tilde{a}\notin\Z$. Together
we have that $\tilde{a}$ satisfies $F$. \\
So $\Sigma(x)$ is consistent. Let $a^{\tg}\in\mathcal{M}$ be a realization
of $\Sigma(x)$. As $tp(a^{\tg}/\emptyset)=tp(a/\emptyset)$, there
is an automorphism of $L_{p}$-structures $\sigma\in Aut(\mathcal{M}/\emptyset)$
such that $\sigma(a)=a^{\tg}$. Let $b^{\tg}=\sigma(b)$. So $tp(b^{\tg}/\emptyset)=tp(b/\emptyset)$
and $\psi(x,b^{\tg})$ is equivalent to $v(x)\geq v(a^{\tg})$. As
$v(a^{\tg})=v(c)$, we have what we want. \\
Proof of (2). Let $\xi(w,w^{\tg})$ be the formula defined by $\alpha(w)\rightarrow\alpha(w^{\tg})$.
By (\ref{enu:finishing_claim_valuation_part_1}), $\xi(a,c)$ holds
for any $a,c$ such that $v(a)$ and $v(c)$ are nonstandard, so the
result follows by compactness.
\end{proof}
~

Now, let $\theta(z)$ be the formula expressing that $(\phi(x,z) ,+)$ is
a subgroup. By \lemref{Characterizing_definable_subgroups} there
are $n_{1},\dots,n_{k}$, having $\mbox{gcd}(n_{i},p)=1$ for each $i$,
such that for all $c\in\mathcal{M}$ for which $\theta(c)$ holds,
$\phi(x,c)$ is equivalent to a formula of the form $D_{n_{i}}(x)\wedge v(x)\geq v(d)$
for some $i$ and some $d\in\mathcal{M}$. As $(N,+,0,|_{p})$ is
an elementary substructure, if $c\in N$ then there exists such $d\in N$.
Let $n=\prod_{i}n_{i}$, and let $\psi(x,z)$ be the formula $\phi(nx,z)$.
Then for all $c\in\mathcal{M}$ for which $\theta(c)$ holds, $\psi(x,c)$
is equivalent to $v(x)\geq v(d)$, for the same $d$ corresponding
to $\phi(x,c)$ (as $v(n)=0$). 

Let $K\in\N$ be as given by the claim for $\psi(x,z)$ and $\theta(z)$,
and let $\alpha(w)$ and $\chi(w)$ be as in the claim. 
We have that $\psi(x,b)$ is equivalent
to $v(x)\geq\gamma$. In particular, the formula $\rho(z)$ defined
by 
\[
\theta(z)\wedge\exists w(v(w)\geq K\wedge\forall x(\psi(x,z)\leftrightarrow v(x)\geq v(w)))
\]
is satisfied by $b$. Since $\rho(z)$ contains no parameters, there
exists $c\in N$ such that $(N,+,0,|_{p})\vDash\rho(c)$. So $\theta(c)$
holds and there exists $d\in N$ such that $v(d)\geq K$ and $\psi(x,c)$
is equivalent to $v(x)\geq v(d)$. So $(N,+,0,|_{p})\vDash\alpha(d)$.
As $v(d)\geq K$, by the claim we have $\mathcal{M}\vDash\chi(d)$.
Since $\chi(w)$ contains no parameters, also $(N,+,0,|_{p})\vDash\chi(d)$.
Hence, as $v_{p}$ is surjective, for every $\gamma\in\Gamma(N)$
such that $\gamma\geq K$ there exists $c_{\gamma}\in N$ such that
$\theta(c_{\gamma})$ holds and $\psi(x,c_{\gamma})$ is equivalent
to $v(x)\geq\gamma$.

Let $\delta(x,y)$ be the formula 
\[
\bigwedge_{k=1}^{K-1}(D_{p^{k}}(x)\rightarrow D_{p^{k}}(y))\wedge \forall z(\theta(z)\rightarrow(\psi(x,z)\rightarrow\psi(y,z)))\mbox{.}
\]
Then $\delta(x,y)$ is $L$-definable over $\emptyset$, and we claim
that it defines $v(x)\leq v(y)$ in $\cN$: Let $a_{1},a_{2}\in N$,
and suppose $v(a_{1})\leq v(a_{2})$. Then of course $\bigwedge_{k=1}^{K-1}(D_{p^{k}}(a_{1})\rightarrow D_{p^{k}}(a_{2}))$.
Let $c\in N$ such that $\theta(c)$. Then there exists $d\in N$
such that $\psi(x,c)$ is equivalent to $v(x)\geq v(d)$, and therefore
also $\psi(a_{1},c)\rightarrow\psi(a_{2},c)$. So we have $\delta(a_{1},a_{2})$.
On the other hand, suppose $\delta(a_{1},a_{2})$. If $v(a_{1})\leq K-1$,
then by $\bigwedge_{k=1}^{K-1}(D_{p^{k}}(a_{1})\rightarrow D_{p^{k}}(a_{2}))$
we get $v(a_{1})\leq v(a_{2})$. Otherwise, we have that $\gamma:=v(a_{1})\geq K$
and hence $\psi(a_{1},c_{\gamma})$. From $\forall z(\theta(z)\rightarrow(\psi(a_{1},z)\rightarrow\psi(a_{2},z)))$,
as $\theta(c_{\gamma})$ holds, we get in particular $\psi(a_{1},c_{\gamma})\rightarrow\psi(a_{2},c_{\gamma})$,
and therefore we get $\psi(a_{2},c_{\gamma})$, which means $v(a_{2})\geq\gamma=v(a_{1})$.
Therefore, $v(x)\leq v(y)$ is definable over $\emptyset$ in $\cN$.
\end{proof}
~

Combined with \factref{conant_pillay_stable_expansions_finite_dp_rank}
and \thmref{our_result_dp_rank}, we obtain \thmref{our_result_full_p_adic_elementary_extension_version}
and \corref{our_result_full_p_adic}:
\begin{proof}[Proof of \thmref{our_result_full_p_adic_elementary_extension_version}]
Identical to the proof of \corref{conant_result_order_elementary_extension_version}
from \thmref{conant_result_order_unstable_elementary_extension_version}.
\end{proof}
~
\begin{proof}[Proof of \corref{our_result_full_p_adic}]
Identical to the proof of \thmref{conant_result_order} from \corref{conant_result_order_elementary_extension_version}.
\end{proof}
~

\section{Intermediate structures
in elementary extensions}\label{sec:counterexamples_elementary_extensions}

In this section, we show that \factref{conant_pillay_stable_expansions_finite_dp_rank},
\thmref{conant_result_order} and \corref{our_result_full_p_adic}
are no longer true if we replace $\mathbb{Z}$ by an elementarily
equivalent structure. In the case of \corref{our_result_full_p_adic},
there are both stable and unstable counterexamples. For \thmref{conant_result_order}
there are unstable counterexamples, but we do not know whether there
are stable ones.

~

For each of the above we give a family of counterexamples. 
~
\begin{prop}
\label{prop:counterexample_stable_dp_minimal_expansion_of_N_equiv_Z}Let
$(N,+,0,1,|_{p})$ be a nontrivial elementary extension of $(\mathbb{Z},+,0,1,|_{p})$,
let $b\in N$ be such that $\gamma:=v_{p}(b)$ is nonstandard, and let
$B=\{a\in N\,:\, v_{p}(a)\geq\gamma\}$. Then
$(N,+,0,1,B)$ is a stable proper expansion of $(N,+,0,1)$ of dp-rank
$1$. In particular, it is a proper reduct of $(N,+,0,1,|_{p})$.\end{prop}
\begin{proof}
 It is clear that $(N,+,0,1,B)$ is a proper expansion of $(N,+,0,1)$,
and, as a reduct of $(N,+,0,1,|_{p})$, by \thmref{our_result_dp_rank}
it is of dp-rank $1$. It remains to show stability. In \cite[Example 0.3.1 and Theorem 4.2.8]{Wagner_1997},
Wagner defines an abelian structure to be an abelian group together
with some predicates for subgroups of powers of this group. Every
module is an abelian structure. Wagner proves that, as with modules,
in an abelian structure every definable set is equal to a boolean
combination of cosets of $acl(\emptyset)$-definable subgroups. As
a consequence, every abelian structure is stable. Under the assumptions
of \propref{counterexample_stable_dp_minimal_expansion_of_N_equiv_Z},
$B$ is a subgroup of $N$, so $(N,+,0,1,B)$ is an abelian structure, hence stable.
 \end{proof}

~

Let $(N,+,0,1,|_{p})$ be a nontrivial elementary extension of $(\mathbb{Z},+,0,1,|_{p})$. For $\gamma\in \Gamma$ we define $$C_\gamma =\set{(a,b)\in N^{2}\,:\, v_{p}(a)\leq\gamma\wedge v_{p}(b)\leq\gamma\wedge v_{p}(a)\leq v_{p}(b)}.$$

\begin{prop}
\label{prop:counterexample_unstable_valuation_intermediate}
There is a nontrivial elementary extension $(N,+,0,1,|_{p})$ of $(\mathbb{Z},+,0,1,|_{p})$ and a nonstandard $\gamma\in \Gamma$ such that $(N,+,0,1,C_\gamma)$ is an unstable expansion of $(N,+,0,1)$ and a proper reduct of $(N,+,0,1,|_p)$.
\end{prop}
\begin{proof}
For each $m\in\N$, let 
\begin{align*}
C_m &= \set{(a,b)\in\Z^{2}\,:\, a|_{p}p^{m}\wedge b|_{p}p^{m}\wedge a|_{p}b}\\
&=\{(a,b)\in\Z^{2}\,:\,\neg D_{p^{m+1}}(a)\wedge\neg D_{p^{m+1}}(b)\wedge\bigwedge_{i=1}^{m}(D_{p^{i}}(a)\rightarrow D_{p^{i}}(b))\}.
\end{align*}
Let $\cZ_{m}=(\Z,+,0,1,|_{p},C_{m})$. Let $\mathcal{U}$ be a
non-principal ultrafilter on $\N$, and let $\cN=\prod_{\mathcal{U}}\cZ_{m}=(N,+,0,1,|_{p},C)$.
Let $\psi(z)$ be the formula $\forall x,y(C(x,y)\leftrightarrow x|_{p}z\wedge y|_{p}z\wedge x|_{p}y)$.
For any $m\geq k\geq0$, $\cZ_{m}\models \exists z\psi(z) \wedge\forall z(\psi(z)\rightarrow p^{k}|_{p}z)$,
and therefore also $\cN\models \exists z\psi(z)\wedge\forall z(\psi(z)\rightarrow p^{k}|_{p}z)$.
Hence there exists $c\in N$ such that $\gamma:=v_{p}(c)$ is nonstandard
and $C=C_\gamma$.

Suppose for a contradiction that $|_{p}$ is definable in $(N,+,0,1,C)$.
Then there is a formula $\phi(x,y,z)$ in the language of $(N,+,0,1,C)$
with $|x|=|y|=1$, and there is $d\in N$, such that $\cN\models\forall x,y(x|_{p}y\leftrightarrow\phi(x,y,d))$.
Let $(d_{m})_{m\in\N}$ be a representative for $d$ mod $\mathcal{U}$.
Then there exists $m\in\N$ such
that $\cZ_{m}\models\forall x,y(x|_{p}y\leftrightarrow\phi(x,y,d_{m}))$.
Hence $|_{p}$ is definable in $(\Z,+,0,1,C_{m})$. But $C_{m}$ is
definable in $(\Z,+,0,1)$, a contradiction. 

It is clear
that $(N,+,0,1,C)$ is an unstable proper expansion of $(N,+,0,1)$.
\end{proof}
~
\begin{prop}
\label{prop:counterexample_unstable_order_intermediate}
There is a nontrivial elementary extension $(N,+,0,1,<)$ of $(\Z,+,0,1,<)$, and a positive nonstandard $b\in N$, such that $(N,+,0,1,[0,b])$ is an unstable expansion of $(N,+,0,1)$ and a proper reduct of $(N,+,0,1,<)$.
\end{prop}
\begin{proof}
For each $m\in\N$, let $B_{m}=[0,m]=\{0,1,\dots,m\}$, and let
$\cZ_{m}=(\Z,+,0,1,<,B_{m})$. Let $\cN=\prod_{\mathcal{U}}\cZ_{m}=(N,+,0,1,<,B)$
be the ultraproduct of $\{\cZ_{m}\}_{m}$ with respect to some nonprincipal ultrafilter $\mathcal{U}$ over $\N$.
For any $m\geq k\geq 0$, $\cZ_{m}\models\exists!x(\forall y(B_m(y)\leftrightarrow0\leq y\leq x)\wedge x\geq\underline{k})$
and therefore also $\cN\models\exists!x(\forall y(B(y)\leftrightarrow0\leq y\leq x)\wedge x\geq\underline{k})$.
Hence there exists a positive nonstandard element $b\in N$ such that
$B=[0,b]$. 
Suppose for a contradiction that $<$ is definable in $(N,+,0,1,B)$.
Then there is a formula $\phi(x,y,z)$ in the language of $(N,+,0,1,B)$
with $|x|=|y|=1$, and there is $c\in N$, such that $\cN\models\forall x,y(x<y\leftrightarrow\phi(x,y,c))$.
Let $(c_{m})_{m\in\N}$ be a representative for $c$ mod $\mathcal{U}$.
Then there exists $m\in\N$ such
that $\cZ_{m}\models\forall x,y(x<y\leftrightarrow\phi(x,y,c_{m}))$.
Hence $<$ is definable in $(\Z,+,0,1,B_{m})$, a contradiction. 
 It is clear
that $(N,+,0,1,B)$ is a proper expansion of $(N,+,0,1)$. The formula
$B(y-x)$ defines the ordering on $B$, so this structure is unstable.
\end{proof}
~
\begin{rem}
The conclusions of Propositions~\ref{prop:counterexample_unstable_valuation_intermediate} and~\ref{prop:counterexample_unstable_order_intermediate} in fact hold for \emph{any} nontrivial elementary extension and nonstandard $\gamma\in \Gamma$ or positive nonstandard $b\in N$, respectively. In both cases, this can be proved by showing that any structure of this form is sufficiently elementarily equivalent to the specific examples in Propositions~\ref{prop:counterexample_unstable_valuation_intermediate} and~\ref{prop:counterexample_unstable_order_intermediate}. We leave this as an exercice.
\end{rem}

\subsection*{Acknowledgements: }

This paper is part of the first named author\textquoteright s M.Sc., written
at the Hebrew University of Jerusalem under the supervision of Itay
Kaplan, and part of the second named author\textquoteright s Ph.D., written
at Université Claude Bernard Lyon 1, at the time under the supervision
of Pierre Simon. We would like to thank both of them for their useful
suggestions and feedback. In particular, we thank Pierre Simon for
introducing the idea for proving Conant\textquoteright s theorem,
and we thank Itay Kaplan for suggesting the proofs in \textsection ~\ref{sec:counterexamples_elementary_extensions},
and for carefully reading this paper and pointing out mistakes and
inaccuracies. Finally we would like to thank the referee for their tremendous work. Their numerous corrections, suggestions, and advices played a crucial role in making this article readable.

\bibliographystyle{plain}
\bibliography{intermediate_structures}

\end{document}